\documentclass[final,notitlepage,12pt,reqno]{amsart}
\usepackage{graphicx,extarrows}
\usepackage{calc}
\usepackage{url}

\makeatletter
\let\I\@undefined
\makeatother

\usepackage{geometry}
\usepackage{enumitem}[2011/09/28]
\setenumerate{align=left, leftmargin=0pt,labelsep=.5em, labelindent=0\parindent,listparindent=\parindent,itemindent=*}
\usepackage{array}

\usepackage{caption}[2012/02/19]

\usepackage{colonequals}

\usepackage{amssymb,amsmath}
\usepackage{amsfonts}

\usepackage[french,german,english]{babel}

\newcommand{\pL}{\text{\textit\L}}
\newcommand{\qoppa}{\kern-.1em\rotatebox{180}{\raisebox{-0.42em}{$'$}}\kern-.2em{o}}

\DeclareMathOperator{\Span}{span}
\DeclareMathOperator{\lcm}{lcm}
\DeclareMathOperator{\Li}{Li}

\DeclareMathOperator{\D}{d}
\DeclareMathOperator{\I}{Im}
\DeclareMathOperator{\RE}{Re}

\def\pd{\text{\textit\pounds}}

\def\eor{\hfill$ \square$}

\newcolumntype{L}{>{$}l<{$}}
\newcolumntype{C}{>{$}c<{$}}
\newcolumntype{R}{>{$}r<{$}}

\theoremstyle{plain}
\newtheorem{theorem}{Theorem}[section]

\newtheorem{corollary}[theorem]{Corollary}

\newenvironment{remark}[1][Remark]{\begin{trivlist}
\item[\hskip \labelsep {\bfseries #1}]}{\end{trivlist}}

\theoremstyle{definition}

\numberwithin{equation}{section}

\usepackage{Baskervaldx}
\usepackage[baskervaldx]{newtxmath}
\usepackage[cal=cm,scr=rsfs,frak=euler,bb=ams]{mathalfa}

\usepackage{color}

\def\l{\left}
\def\r{\right}
\def\bg{\bigg}
\def\({\bg(}
\def\){\bg)}
\def\t{\text}
\def\f{\frac}

\def\bi{\binom}

\begin{document}

\pagenumbering{roman}
\selectlanguage{english}
\title{Evaluations of $ \sum_{k=1}^\infty \frac{x^k}{k^2\binom{3k}{k}}$ and related series}
\author{Zhi-Wei Sun}\address{(Zhi-Wei Sun) Department of Mathematics, Nanjing
University, Nanjing 210093, People's Republic of China}
\email{{\tt zwsun@nju.edu.cn}
\newline\indent
{\it Homepage}: {\tt http://maths.nju.edu.cn/\lower0.5ex\hbox{\~{}}zwsun}}
 \author{Yajun Zhou
}
\address{(Yajun Zhou) Program in Applied and Computational Mathematics (PACM), Princeton University, Princeton, NJ 08544} \email{yajunz@math.princeton.edu}\curraddr{\textrm{} \textsc{Academy of Advanced Interdisciplinary Studies (AAIS), Peking University, Beijing 100871, P. R. China}}\email{yajun.zhou.1982@pku.edu.cn}
\date{\today}\thanks{\indent\textit{2020 Mathematics Subject Classifiction}:  Primary 05A10, 11M32; Secondary 11B65, 33B15\\\indent *The first author was supported by the Natural Science Foundation of China (grant no. 12371004), and the second author was supported in part  by the Applied Mathematics Program within the Department of Energy
(DOE) Office of Advanced Scientific Computing Research (ASCR) as part of the Collaboratory on
Mathematics for Mesoscopic Modeling of Materials (CM4).}

\maketitle

\begin{abstract}
     We perform polylogarithmic reductions for several classes of infinite sums motivated by Z.-W. Sun's related works in 2022--2023. For certain choices of parameters, these series can be  expressed by cyclotomic multiple zeta values of levels $4$, $5$, $6$, $7$, $8$, $9$, $10$, and $12$. In particular, we obtain closed forms of the series
$$\sum_{k=0}^\infty\frac{x_0^k}{(k+1)\binom{3k}k}
\ \ \text{and}\ \ \sum_{k=1}^\infty\frac{x_0^k}{k^2\bi{3k}k}$$
for any $x_0\in(-27/4,27/4)$.
     \end{abstract}

\pagenumbering{arabic}

\section{Introduction}

Motivated by the identity
$$\sum_{k=0}^\infty\f{25k-3}{2^k\bi{3k}k}=\f{\pi}2$$
observed by R. W. Gosper in 1974 (cf. \cite{AKP2003}),
 Z.-W. Sun \cite{Sun2022ab} determined in 2022 the values of
\begin{align}\sum_{k=1}^\infty\frac {k^rx^k}{\bi{3k}k}\ \left(-\frac{27}4<x<\frac{27}4\right)\ \text{ and } \ \sum_{k=1}^\infty\frac {k^rx^k}{\binom{4k}{2k}}\ (-16<x<16)\end{align}
for $r=0,\pm1$. For example, \cite[Theorem 1.1]{Sun2022ab} implies that
\begin{align}
\begin{split}
\sum_{k=0}^\infty\f{x^{3k}}{(x-1)^k\bi{3k}k}=&\ \f{27(1-x)}{(x+3)(2x-3)^2}+\f{3x(x-1)}{(2x-3)^3}\log(1-x)
\\&\ +\f{2x(x-1)(x^2-12x+9)q(x)}{(x+3)(2x-3)^3\sqrt{(1-x)(x+3)}}
\end{split}
\end{align}
for all $x\in(-3,c)$,
where
\begin{equation}\label{c}c\colonequals \f32\l[\big(1+\sqrt2\big)^{1/3}-\big(1+\sqrt2\big)^{-1/3}\r]=0.8941\cdots
\end{equation}
and
 \begin{equation}\label{q(x)} q(x)\colonequals \arg\frac{1+x-i \sqrt{(1-x) (3+x)}}{1-x-i \sqrt{(1-x) (3+x)}}=\begin{cases}\tan^{-1}\Big(\f x{x+2}\sqrt{\f{3+x}{1-x}}\Big)&\t{if}\ -2<x<1,
 \\-\f{\pi}2&\t{if}\ x=-2,
 \\\tan^{-1}\Big(\f x{x+2}\sqrt{\f{3+x}{1-x}}\Big)-\pi&\t{if}\ -3<x<-2.
 \end{cases}
 \end{equation}Here, ``$ \tan^{-1}$'' denotes the inverse tangent function.
 As observed in \cite{Sun2022ab},
if $x_0\in(-27/4,27/4)$ then $x^3=x_0(x-1)$ for some $x\in(-3,c)$.

 In 2023, Sun \cite{Sun2023} evaluated \begin{align}\sum_{k=1}^\infty\f{x_0^k}{(2k-1)\bi{3k}k},\ \sum_{k=0}^\infty\f{x_0^k}{(3k+1)\bi{3k}k},
\ \t{and}\ \sum_{k=0}^\infty\f{x_0^k}{(3k+2)\bi{3k}k}\end{align}
for any $x_0\in(-27/4,27/4)$, and
\begin{align}\sum_{k=1}^\infty\f{k^rx_0^k}{(2k-1)\bi{4k}{2k}},\ \sum_{k=1}^\infty\f{k^rx_0^k}{(4k+1)\bi{4k}{2k}}
,\ \t{and}\ \sum_{k=1}^\infty\f{k^rx_0^k}{(4k+3)\bi{4k}{2k}}\end{align}
for any $r\in\{0,1\}$ and $x_0\in(-16,16)$.

Sun \cite{Sun2023} proved that
\begin{equation}\label{k+1}\sum_{k=0}^\infty\f1{(k+1)2^k\bi{3k}k}=3\log^2 2-\f{\pi^2}4+\f{8\pi-21\log2}5
\end{equation}
and
\begin{equation}\label{8/3-k+1}\sum_{k=0}^\infty\f{8^k}{(k+1)3^k\bi{3k}k}=\f{9\log^23-3\pi^2}{16}+\f{11\sqrt3\pi}{14}-\f {9\log3}7.
\end{equation}
Motivated by this, we first evaluate the series
\begin{align}\sum_{k=0}^\infty\f{x_0^k}{(k+1)\bi{3k}k}
\ \t{and}\ \sum_{k=1}^\infty\f{x_0^k}{k^2\bi{3k}k}\end{align}
for any $x_0\in(-27/4,27/4)$. Namely, we obtain the following two general theorems.

 \begin{theorem}\label{thm:3k_q} We have
 \begin{align}
\begin{split}
\sum_{k=0}^\infty\frac{x^{3k}}{(k+1)(x-1)^{k}\binom{3k}k}={}&\frac{2 (1-x)q(x)^2}{x^3}-\frac{ x^2+15 x-18}{x^2  (2 x-3)}\sqrt{\frac{1-x}{3+x}}q(x)\\{}&-\frac{3 (1-x) \log ^2(1-x)}{2 x^3}+\frac{3 (x-6) (x-1) \log (1-x)}{2 x^2 (2 x-3)}\label{eq:conj1.1}
\end{split}\end{align} for any $x\in(-3,c)\smallsetminus\{0\}$.
\end{theorem}

\begin{theorem}\label{thm:conj1.2}
For $ x\in(-3,c)$, we have \begin{align}\sum_{k=1}^\infty\f{x^{3k}}{k^2(x-1)^k\bi{3k}k}
=\f23q(x)^2-\f12\log^2(1-x).\label{eq:conj1.2}
\end{align}\end{theorem}

Sun \cite{Sun2023} proved that \begin{equation}\label{4^k}\sum_{k=0}^\infty\f{4^k}{(k+1)\bi{4k}{2k}}=\f{20\pi-3\pi^2}8-\f7{\sqrt3}\log\f{\sqrt3+1}{\sqrt3-1}
+\f 32\log^2\f{\sqrt3+1}{\sqrt3-1}.
\end{equation}
Motivated by this, we evaluate
\begin{align}\sum_{k=1}^\infty\f{x_0^k}{(k+1)\bi{4k}{2k}}
 \text{ and } \sum_{k=0}^\infty\f{x_0^k}{k^2\bi{4k}{2k}}\end{align}
for all $x_0\in(-16,16)$, arriving at  the general results in the next two  theorems.

\begin{theorem}\label{thm:conj1.3}\begin{enumerate}[leftmargin=*,  label=\emph{(\alph*)},ref=(\alph*),
widest=d, align=left] \item
For $ X>\frac{1}{4}$, we have \begin{align}\begin{split}
\sum_{k=0}^\infty\frac{1}{(k+1)X^{2k}\binom{4k}{2k}}={}&\frac{4X (12X-1) \cot ^{-1}\sqrt{4X-1}}{\sqrt{4X-1}}-24X^2 \left( \cot ^{-1}\sqrt{4X-1} \right)^2\\{}&-\frac{4X(12 X+1) \tanh ^{-1}\frac{1}{\sqrt{4 X+1}}}{\sqrt{4X+1}}+24X^2 \left(\tanh ^{-1}\frac{1}{\sqrt{4X+1}}\right)^2.\end{split}\end{align}
\item If $ x>\frac{1+\sqrt{2}}{2}$ or $ x<\frac{1-\sqrt{2}}{2}$,  then we have \begin{align}
\begin{split}&
\sum_{k=0}^\infty\frac{2^{2k}}{(k+1)[x(1-x)]^{k}\binom{4k}{2k}}\\={}&\frac{2 (1-x) (6 x-1)}{1-2 x} \sqrt{x} \tanh ^{-1}\frac{1}{\sqrt{x}}+\frac{2 x (6 x-5)}{1-2 x} \sqrt{1-x} \tanh ^{-1}\frac{1}{\sqrt{1-x}}\\{}&+3 x(1-x)  \left(\tanh ^{-1}\frac{1}{\sqrt{x}}\right)^2+3 x(1-x)  \left(\tanh ^{-1}\frac{1}{\sqrt{1-x}}\right)^2.
\end{split}
\end{align}\end{enumerate}

For $x_0\in(-16,0)$, it is easy to see that $x_0=4/(x(1-x))$ for some $x>(1+\sqrt{2})/2$ or $x<(1-\sqrt2)/2$.

\end{theorem}\begin{theorem} \label{thm:conj1.4} If $x>(1+\sqrt2)/2$ or $x<(1-\sqrt2)/2$, then we have
\begin{align}
\sum_{k=1}^\infty\f{4^k}{k^2[x(1-x)]^k\bi{4k}{2k}}
=-2\left(\tanh ^{-1}\frac{1}{\sqrt{x}}\right)^{2}-2\left(\tanh ^{-1}\frac{1}{\sqrt{1-x}}\right)^2.
\label{eq:4k2ksymm}\end{align}
\end{theorem}

For $r\in\mathbb Z_{>0}=\{1,2,\ldots\}$, the $r$-th order harmonic numbers are defined by \begin{align} \mathsf H_m^{(r)}\colonequals\sum_{k=1}^m\frac1{k^{r}}
 \ (m\in\mathbb Z_{\ge0}=\{0,1,2,\ldots\}).\end{align}
 Those $\mathsf H_m=\mathsf H_m^{(1)}\ (m=0,1,2,\ldots)$ are the usual harmonic numbers.
 Infinite series with summands involving binomial coefficients in their denominators and harmonic numbers   in their numerators, are known as ``inverse binomial sums'' to high energy physicists \cite{DavydychevKalmykov2004,Weinzierl2004bn}.

Sun \cite{Sun2023} evaluated the series
\begin{align}\sum_{k=0}^\infty\f{\mathsf H_{3k}-\mathsf H_k}{(2k-1)2^k\bi{3k}k},\ \sum_{k=0}^\infty\f{\mathsf H_{3k+1}-\mathsf H_k}{(3k+1)2^k\bi{3k}k},
\ \sum_{k=0}^\infty\f{\mathsf H_{2k}-\mathsf H_k}{(3k+1)2^k\bi{3k}k},\text{ and } \sum_{k=0}^\infty\f{\mathsf H_{2k}-\mathsf H_k}{(3k+2)2^k\bi{3k}k}.\end{align}
We present a more general approach in \S\ref{sec:polylog_red} of the current work, accommodating to selected  ($r$-th order)  harmonic numbers and  rational functions of $k$ in the summands. To execute such analysis, we will need
generalized polylogarithms (GPLs), which are defined recursively by \cite[(2.1) and (2.2)]{Frellesvig2016}: \begin{align}\notag\\[-12pt]
\left\{\begin{array}{@{}l}
G(\alpha_{1},\dots,\alpha_n;z)\colonequals\smash[t]{\displaystyle\int_0^z\frac{\D x}{x-\alpha_1}G(\alpha_2,\dots,\alpha_n;x)},\quad\text{if }|\alpha_1|+\cdots+|\alpha_n|\neq0,\ \\[15pt]
\smash[b]{G(\underset{n }{\underbrace{0,\dots,0 }};z)}\colonequals\smash[b]{\dfrac{\log^nz}{n!}},\quad G(-\!\!-;z)\colonequals1. \\
\end{array}\right.\label{eq:GPL_rec}\\[-6pt]\notag
\end{align}In the two theorems below (to be proved in \S\ref{subsec:GPLchar}), we require the following $\mathbb Q$-vector space for suitably chosen $ z\in\mathbb C$, $ k,k_*\in \mathbb Z_{>0}$, and $ S,S_*\subset\mathbb C$:\begin{align}
\mathfrak G_{k;k_*}^{(z)}[S;S_*]\colonequals{}&\sum_{j=k_*}^k\Span_{\mathbb Q}\left\{ G(\alpha_1,\dots,\alpha_k;z)\left|\begin{smallmatrix}\alpha_\ell\in\{0,1\}\cup  S,\ell\in\mathbb Z\cap([1,k]\smallsetminus\{j\})\\{\alpha_j}\in{S_*}\end{smallmatrix}\right. \right\}.
\end{align} According to the original definition of GPLs, the expression $ G(\alpha_1,\dots,\alpha_k;z)$ diverges if $\alpha_1=z$. Later on,  we will tacitly employ Panzer's logarithmic regularization procedure \cite[\S2.3]{Panzer2015} in the  working definition of ``$ G(\alpha_1,\dots,\alpha_k;\alpha_{1})$'', without altering the notation for $ \mathfrak G_{k;k_*}^{(z)}[S;S_*]$.

Recall the constant $c$ from \eqref{c}, the function $ q(x)$ from \eqref{q(x)}.
We also define
\begin{align}  \tau^\pm(x)=\frac{1-x\pm i \sqrt{(1-x) (3+x)}}{2 x}\end{align}
and
\begin{align} \Li_2(z)\colonequals -G\left(0,\frac1z;1\right)\ \t{for}\ z\in\mathbb C\smallsetminus[\{0\}\cup(1,\infty)].\end{align}
 Note that  $ |x^3|\leq\frac{3^{3}}{2^{2}}|1-x|$ \big(resp.\  $ |x^3|<\frac{3^{3}}{2^{2}}|1-x|$\big) when $ x\in[-3,c]$ [resp.\ $ x\in(-3,c)$].

\begin{theorem}\label{thm:3kGPL}  \begin{enumerate}[leftmargin=*,  label=\emph{(\alph*)},ref=(\alph*),
widest=d, align=left] \item
If $ x\in[-3,c]\smallsetminus\{0\}$, $r\in\mathbb Z_{>0} $,  then we have\begin{align}
\sum_{k=1}^\infty\frac{a_{k}}{k^{r+1}\binom{3k}k}\left( \frac{x^{3}}{x-1} \right)^{k}\in\mathfrak{G}^{(1)}_{r+2;r+1}\left[\varnothing;\left\{ \frac{1}{x},-\tau^+ (x),-\tau^-(x)\right\} \right],\label{eq:3k_inv_sum1}
\end{align}when the sequence $ (a_k)_{k\in\mathbb Z_{>0}}$ is a member of $ \big\{\big(\frac1k\big)_{k\in\mathbb Z_{>0}},$  $(\mathsf H_{k}-\mathsf H_{3k})_{k\in\mathbb Z_{>0}},$ $(\mathsf H_{2k}-\mathsf H_{3k})_{k\in\mathbb Z_{>0}}\big\}$.
 If $ x\in[-3,c]\smallsetminus\{0\}$ and $r,r_{1},\dots, r_M\in\mathbb Z_{>0} $,  then we have\begin{align}
\begin{split}&
\sum_{k=1}^\infty\frac{\prod_{j=1}^M\mathsf H_{k}^{(r_{j})}}{k^{r+1}\binom{3k}k}\left( \frac{x^{3}}{x-1} \right)^{k}\\\in{}&{\mathfrak{G}}^{(1)}_{r+1+\sum_{j=1}^M r_j;r+1+\sum_{j=1}^M r_j}\left[\left\{ \frac{1}{x},-\tau^+ (x),-\tau^-(x)\right\};\left\{ \frac{1}{x},-\tau^+ (x),-\tau^-(x)\right\} \right].
\end{split}\label{eq:H_prod_sum}
\end{align}
\item If $ x\in(-3,c)\smallsetminus\{0\}$, then we have{\allowdisplaybreaks\begin{align}
\begin{split}
\sum_{k=0}^\infty\frac{\mathsf H_{k}-\mathsf H_{2k}}{(3k+1)\binom{3k}k}\left( \frac{x^{3}}{x-1} \right)^{k}={}&-\frac{1-x}{2x (3-2 x)}\left[ q(x)^2+\frac{3 \log ^2(1-x)}{4} \right]\\{}&-\frac{3-x}{2 x (3-2 x)}\sqrt{\frac{1-x}{3+x}}q(x) \log (1-x)\\={}&-\frac{1-x}{2x (3-2 x)}\left[q(x)+\frac{1}{2} \sqrt{\frac{3+x}{1-x}} \log (1-x)\right]\\&{}\times\left[q(x)+\frac{3}{2} \sqrt{\frac{1-x}{3+x}} \log (1-x)\right],\label{eq:Hk-H2k(3k+1)}
\end{split}\\\begin{split}
\sum_{k=0}^\infty\frac{\mathsf H_{k}-\mathsf H_{3k+1}}{(3k+1)\binom{3k}k}\left( \frac{x^{3}}{x-1} \right)^{k}={}&-\frac{1-x}{2x (3-2 x)}\left[ 3 \Li_2(x)+\frac{2 q(x)^2}{3}+\log ^2(1-x) \right]\\&{}+i\frac{3-x}{2 x (3-2 x)}\sqrt{\frac{1-x}{3+x}}\left[\Li_2\left( -\frac{1}{\tau^+(x)} \right)-\Li_2\left( -\frac{1}{\tau^-(x)} \right) \right]\\&{}-\frac{3-x}{2 x (3-2 x)}\sqrt{\frac{1-x}{3+x}}q(x) \log (1-x),\label{eq:Hk-H3k(3k+1)}
\end{split}\\\begin{split}
\sum_{k=0}^\infty\frac{\mathsf H_{k}-\mathsf H_{2k}}{(3k+2)\binom{3k}k}\left( \frac{x^{3}}{x-1} \right)^{k}={}&-\frac{ (3-x) (1-x)}{2 x^2 (3-2 x)}q(x)^2+\frac{3 (3+x) }{x^2}\sqrt{\frac{1-x}{3+x}}q(x)\\{}&+\frac{9 -6x-x^{2}  }{2 x^2 (3-2 x)}\sqrt{\frac{1-x}{3+x}}q(x) \log (1-x)\\{}&+\frac{9 (1-x) }{2 x^2}\log (1-x)-\frac{ 3(3-x) (1-x) }{8 x^2 (3-2 x)}\log ^2(1-x),\label{eq:Hk-H2k(3k+2)}
\end{split}\\\begin{split}
\sum_{k=1}^\infty\frac{\mathsf H_{k}-\mathsf H_{2k-2}}{(2k-1)\binom{3k}k}\left( \frac{x^{3}}{x-1} \right)^{k}={}&-\frac{(1-x) x}{3 (3-2 x)}\left[ q(x)^2+\frac{3 \log ^2(1-x)}{4} \right]\\&{}+\frac{x \left(3-x^2\right) }{3 \sqrt{(1-x) (3+x)} (3-2 x)}q(x) \log (1-x),\label{eq:Hk-H2k(2k-1)}
\end{split}\\\begin{split}
\sum_{k=1}^\infty\frac{\mathsf H_{k}-\mathsf H_{3k-1}}{(2k-1)\binom{3k}k}\left( \frac{x^{3}}{x-1} \right)^{k}={}&-\frac{(1-x) x}{3 (3-2 x)}\left[ 3 \Li_2(x)+\frac{2 q(x)^2}{3}+\log ^2(1-x) \right]\\&{}-\frac{ix \left(3-x^2\right) }{3 \sqrt{(1-x) (3+x)} (3-2 x)}\left[\Li_2\left( -\frac{1}{\tau^+(x)} \right)-\Li_2\left( -\frac{1}{\tau^-(x)} \right) \right]\\&{}+\frac{x \left(3-x^2\right) }{3 \sqrt{(1-x) (3+x)} (3-2 x)}q(x) \log (1-x).\label{eq:Hk-H3k(2k-1)}
\end{split}
\end{align}}
\end{enumerate}\end{theorem}
\begin{theorem}\label{thm:4kGPL}\begin{enumerate}[leftmargin=*,  label=\emph{(\alph*)},ref=(\alph*),
widest=d, align=left] \item
If $ |4x(1-x)|\geq1$  and $r,r_{1},\dots, r_M\in\mathbb Z_{>0} $,  then we have\begin{align}\begin{split}&
\sum_{k=1}^\infty\frac{\prod_{j=1}^M\mathsf H_{k}^{(r_{j})}}{k^{r+1}\binom{4k}{2k}}\left[ \frac{2^{2}}{x(1-x)} \right]^{k}\\\in{}&{\mathfrak{G}}^{(1)}_{r+1+\sum_{j=1}^M r_j;r+1+\sum_{j=1}^M r_j}\left[\left\{ \xi_{0,0}(x),\xi_{0,1}(x),\xi_{1,0}(x),\xi_{1,1}(x)\right\};\right.\\&{}\left.\left\{ \xi_{0,0}(x),\xi_{0,1}(x),\xi_{1,0}(x),\xi_{1,1}(x)\right\} \right],\label{eq:H_prod_sum4k}
\end{split}
\end{align}where $ \xi_{\ell ,m}(x)\colonequals\frac{1+(-1)^{\ell} \sqrt{x}+(-1)^{m}\sqrt{1-x}}{2}$.
\item If $ |4x(1-x)|>1$, then we have {\allowdisplaybreaks\begin{align}\begin{split}&
\sum_{k=1}^\infty\frac{\mathsf H_{k}}{(4k+1)\binom{4k}{2k}}\left[ \frac{2^{2}}{x(1-x)} \right]^{k}\\={}&-{}\sum_{\ell,\ell',m,m'\in\{0,1\}}\frac{\sqrt{(1-x)x}\big[(-1)^{\ell}\sqrt{1-x}-(-1)^m\sqrt{x}\big]G(\xi_{\ell,m}(x),\xi _{\ell',m'}(x);1)}{4 (2 x-1)},
\end{split}\label{eq:Hk(4k+1)}
\\\begin{split}&\sum_{k=1}^\infty\frac{\mathsf H_{k}}{(4k+3)\binom{4k}{2k}}\left[ \frac{2^{2}}{x(1-x)} \right]^{k}\\={}&\sum_{\ell,\ell',m,m'\in\{0,1\}}\frac{2 [(1-x) x]^{3/2}G(\xi_{\ell,m}(x),\xi_{\ell',m'}(x);1)}{(-1)^{\ell }\sqrt{1-x} +(-1)^m \sqrt{x}-2 \sqrt{(1-x) x}-(-1)^{\ell +m }}\\{}&+\sum_{\ell,\ell',m,m'\in\{0,1\}}\frac{\sqrt{(1-x)x}\big[(-1)^{\ell}\sqrt{1-x}-(-1)^m\sqrt{x}\big]G(\xi_{\ell,m}(x),\xi_{\ell',m'}(x);1)}{4 (2 x-1)},\label{eq:Hk(4k+3)}
\end{split}\end{align}} where the GPLs are defined recursively by \eqref{eq:GPL_rec}.
\end{enumerate}
\end{theorem}

Note that  in the definition of Goncharov's \cite{Goncharov1997,Goncharov1998} multiple polylogarithm (MPL)\begin{align}
\Li_{a_1,\dots,a_n}(z_1,\dots,z_n)\colonequals \sum_{\ell_{1}>\dots>\ell_{n}>0}\prod_{j=1}^n\frac{z_{j}^{\ell_{j}}}{\ell_j^{a_j}},
\label{eq:Mpl_defn}\end{align}the right-hand side  converges absolutely for  $ a_1,\dots,a_n\in\mathbb Z_{>0}, \prod_{j=1}^m|z_j|<1,m\in\mathbb Z\cap[1,n]$. The GPLs  give analytic continuations of MPLs, via the following  MPL-GPL correspondence (see \cite[(1.3)]{Panzer2015} or \cite[(2.5)]{Frellesvig2016}){\begin{align} \Li_{a_1,\dots,a_n}(z_1,\dots,z_n)=(-1)^{n}G\left(\smash[b]{\underset{a_1-1 }{\underbrace{0,\dots,0 }}},\frac{1}{z_{1}},\smash[b]{\underset{a_2-1 }{\underbrace{0,\dots,0 }}},\frac{1}{z_{1}z_2},\dots,\smash[b]{\underset{a_n-1 }{\underbrace{0,\dots,0 }}},\frac{1}{\prod_{j=1}^nz_j};1\right)\label{eq:MPL_GPL}\\[-6pt]\notag\end{align}}when  $ \prod_{j=1}^nz_j\neq0$. As a consequence, for meticulously chosen values of $x$, special cases of Theorems \ref{thm:3kGPL} and \ref{thm:4kGPL} bring us explicit close-form evaluations of infinite series, through  the  $ \mathbb Q$-vector space for cyclotomic multiple zeta values (CMZVs)\footnote{In the literature \cite{Zhao2016MZVbook,SingerZhao2020,Au2022a}, one may find various other notations  for $ \mathfrak Z_k(N) $, including $ \mathsf {CMZV}(k,N)$, $ \mathsf{CMZV}_{k,N}$, and $ \mathsf{CMZV}_k^N$. }  of weight $k\in\mathbb Z_{>0}$ and level $ N\in\mathbb Z_{>0}$ \begin{align}
\begin{split}
\mathfrak Z_{ k}(N)\colonequals{}&\Span_{\mathbb Q}\left\{\Li_{a_1,\dots,a_n}(z_1,\dots,z_n)\left|\begin{smallmatrix}a_1,\dots,a_n\in\mathbb Z_{>0}\\z_{1}^{N}=\dots=z_n^N=1\\(a_1,z_1)\neq(1,1)\\\sum _{j=1}^{n}a_{j}=k\end{smallmatrix}\right. \right\}\\\xlongequal{\text{\eqref{eq:MPL_GPL}}}{}&\Span_{\mathbb Q}\left\{G(z_1,\dots,z_k;1)\left|\begin{smallmatrix}z_1^N,\dots,z_{k}^N\in\{0,1\}\\z_1\neq1,z_{k}\neq0\end{smallmatrix}\right.\right\}.
\end{split}\label{eq:Zk(N)_defn}
\end{align}  The corresponding results are summarized in the two corollaries below (to be proved in \S\ref{subsec:CMZVchar}). \begin{corollary}\label{cor:3kCMZV}Set \begin{align}\mathsf
S_{3,r}\left(a_k;z\right)\colonequals\sum_{k=1}^\infty\frac{a_{k}z^k}{k^{r+1}\binom{3k}k}
\end{align}and\begin{align}
\mathfrak r_{\nu}\colonequals\frac{\left[1-4\cos^2(\nu\pi)\right]^{3}}{-4\cos^2(\nu\pi)}=\frac{16\cos^{3}\left( \nu\pi+\frac{\pi}{6}\right) \cos^{3}\left( \nu\pi-\frac{\pi}{6}\right)}{\cos^2(\nu\pi)}.
\end{align}For $ N\in\{ 4,5,6,7,8,9,10,12\}$ and  $r\in\mathbb Z_{>0} $,   we have \begin{align}
\mathsf S_{3,r}\left( a_k ;\mathfrak r_{m/N}\right)\in{}&\mathfrak Z_{r+2}(N),\label{eq:3kZ(N)}
\end{align}where $m\in\mathbb Z $, $0<|\mathfrak r_{m/N}|<\frac{3^3}{2^2} $, and  $ (a_k)_{k\in\mathbb Z_{>0}}\in\big\{\big(\frac1k\big)_{k\in\mathbb Z_{>0}},$  $(\mathsf H_{k}-\mathsf H_{3k})_{k\in\mathbb Z_{>0}},$ $(\mathsf H_{2k}-\mathsf H_{3k})_{k\in\mathbb Z_{>0}}\big\}$. For any $ N\in\mathbb Z_{>0}$,  with the same constraints on $m$ and $(a_k)_{k\in\mathbb Z_{>0}}$  as above, one has the  following variation on  \eqref{eq:3kZ(N)}:\begin{align}
\mathsf S_{3,r}\left( a_k ;\mathfrak r_{m/N}\right)\in{}&\mathfrak Z_{r+2}(\lcm(6, N)),\label{eq:3kZ(N)'}\tag{\ref{eq:3kZ(N)}$'$}
\end{align}where $ \lcm$ stands for ``least common multiple''.
Furthermore, one has \begin{align}
\mathsf S_{3,r}\left(a_k;\frac{3^{3}}{2^{2}}\right)\in{}&\mathfrak Z_{r+2}(6),\label{eq:3kZ(6)'}
\end{align}when  $ (a_k)_{k\in\mathbb Z_{>0}}\in\big\{\big(\frac1k\big)_{k\in\mathbb Z_{>0}},$  $(\mathsf H_{k}-\mathsf H_{3k})_{k\in\mathbb Z_{>0}},$ $(\mathsf H_{2k}-\mathsf H_{3k})_{k\in\mathbb Z_{>0}},$ $(\mathsf H_k)_{k\in\mathbb Z_{>0}}\big\}$.\end{corollary}\begin{corollary}\label{cor:4kCMZV}Set \begin{align}\mathsf
S_{4,r}\left(a_k;z\right)\colonequals\sum_{k=1}^\infty\frac{a_{k}z^k}{k^{r+1}\binom{4k}{2k}}.
\end{align}For $r\in\mathbb Z_{>0} $,  we have \begin{align}\mathsf
S_{4,r}\big(\mathsf H_k;2^4\big)\in{}&\mathfrak Z_{r+2}(8),\\\mathsf
S_{4,r}\big(\mathsf H_k;2^2\big)\in{}&\mathfrak Z_{r+2}(12).
\end{align}

\end{corollary}As examples for Corollary \ref{cor:3kCMZV}, we have (cf.\ \cite[Remark 1.4]{Zhou2023SunCMZV})\begin{align}
\sum_{k=1}^\infty\frac{1}{k^{2}2^k\binom{3k}k}={}&\displaystyle \frac{\pi ^2}{24}-\frac{\lambda^{2}}{2}\in\mathfrak Z_2(2)\subset\mathfrak Z_2(4),\\\sum_{k=1}^\infty\frac{1}{k^{3}2^k\binom{3k}k}={}&{}\displaystyle -\frac{33 \zeta (3)}{16}+\pi  G+\frac{\lambda ^3}{6}-\frac{\pi ^2 \lambda}{24}\in\mathfrak Z_3(4),\\\sum_{k=1}^\infty\frac{1}{k^{4}2^k\binom{3k}k}={}&{}\displaystyle -\frac{21\Li_{3,1}(-1,1)}{4} - \pi  \I\Li_{2,1}(1,i)+\frac{33  \zeta (3) \lambda}{16}\\&{}-\frac{\pi  G \lambda }{2}-\frac{\lambda ^4}{24}+\frac{\pi ^2 \lambda ^2}{48}+\frac{\pi ^4}{60}\in\mathfrak Z_4(4),\\\sum_{k=1}^\infty\frac{1}{k^{5}2^k\binom{3k}k}={}&\frac{51\Li_{3,1,1}(-1,1,1)}{4} -\frac{1107 \zeta (5)}{128}+\frac{21\Li_{3,1}(-1,1)\lambda}{4} +\frac{4\pi\I\Li_{2,1,1}(1,1,i)}{3}\\&{}+\frac{13\pi\I\Li_4(i)}{3}+\frac{\pi   \lambda\I\Li_{2,1}(1,i)}{3}-\frac{33 \lambda ^2 \zeta (3)}{32}-\frac{191 \pi ^2 \zeta (3)}{192}\\&{}+\frac{\pi  G \lambda ^2}{6} +\frac{2 \pi ^3 G}{9}+\frac{\lambda ^5}{120}-\frac{\pi ^2 \lambda ^3}{144}-\frac{\pi ^4 \lambda }{60}\in\mathfrak Z_4(5).
\end{align}Here, one notes that $ \lambda\colonequals\log 2=-\Li_1(-1)\in\mathfrak Z_1(2)\subset\mathfrak Z_1(4) $, $ \pi i=2[\Li_1(i)-\Li_1(-i)]\in\mathfrak Z_1(4)$, $ G\colonequals \I \Li_2(i)\in i\mathfrak Z_2(4)$, $ \zeta(3)\colonequals \Li_3(1)\in\mathfrak Z_3(1)\subset\mathfrak Z_3(4)$, and  $ \zeta(5)\colonequals \Li_5(1)\in\mathfrak Z_5(1)\subset\mathfrak Z_5(4)$, while $z_j\in\mathfrak Z_j(4) $ and $z_k\in\mathfrak Z_k(4)$ together entail $ z_jz_k\in\mathfrak Z_{j+k}(4)$     \cite[\S1.2]{Goncharov1998}.

 For other small weights $k\in\mathbb Z_{>0}$ and levels $ N\in\mathbb Z_{>0}$, such  $\mathfrak Z_{ k}(N) $ characterizations  boil down to explicitly computable expressions in Au's \texttt{MultipleZetaValues} package \cite{Au2022a}, as illustrated by Tables \ref{tab:Zk(5)3k}--\ref{tab:Zk(8,12)4k} of this article.

\section{Logarithmic reductions of certain inverse binomial sums\label{sec:log_red}}

\subsection{Some integrals related to binomial coefficients\label{subsec:3k4kintrepn}}Using Euler's  gamma function $\Gamma(s) \colonequals\int_0^\infty t^{s-1}e^{-t}\D t$ for $ \RE s>0$, Euler's digamma  function $ \psi^{(0)}(s)\colonequals\frac{\D}{\D s}\log\Gamma(s)$,  and Euler's beta function $ \mathrm B(a,b)=\frac{\Gamma(a)\Gamma(b)}{\Gamma(a+b)}=\int_0^1 t^{a-1}(1-t)^{b-1}\D t$,  one can verify the following integral formulae for  $ k\in\mathbb Z_{\geq0}$:{\allowdisplaybreaks
\begin{align}\frac{1}{(3k+1)\binom{3k}k}={}&\int_0^1[t^{2}(1-t)]^{k}\D t,\label{eq:(3k,k)a'}\\
\frac{\mathsf H_{2k}-\mathsf H_{3k+1}}{(3k+1)\binom{3k}k}={}&\int_0^1[t^{2}(1-t)]^{k}\log t\D t,\label{eq:(3k,k)b'}\\\frac{\mathsf H_{k}-\mathsf H_{3k+1}}{(3k+1)\binom{3k}k}={}&\int_0^1[t^{2}(1-t)]^{k}\log (1-t)\D t,\label{eq:(3k,k)c'}\\\frac{1}{(4k+1)\binom{4k}{2k}}={}&\int_0^1[t(1-t)]^{2k}\D t,\label{eq:(4k,2k)a'}\\\frac{\mathsf H_{2k}-\mathsf H_{4k+1}}{(4k+1)\binom{4k}{2k}}={}&\int_0^1[t^{}(1-t)]^{2k}\log t\D t,\label{eq:(4k,2k)b'}\\\frac{\mathsf H_{2k}-\mathsf H_{4k+1}}{(4k+1)\binom{4k}{2k}}={}&\int_0^1[t^{}(1-t)]^{2k}\log (1-t)\D t.\label{eq:(4k,2k)c'}
\intertext{and the following integral formulae for  $ k\in\mathbb Z_{>0}$:}
\frac{1}{k\binom{3k}k}={}&\int_0^1[t^{2}(1-t)]^{k}\frac{\D t}{1-t},\label{eq:(3k,k)a}\\\frac{\mathsf H_{2k}-\mathsf H_{3k}}{k\binom{3k}k}={}&\int_0^1[t^{2}(1-t)]^{k}\frac{\log t\D t}{1-t},\label{eq:(3k,k)b}\\\frac{\mathsf H_{k-1}-\mathsf H_{3 k}}{k\binom{3k}k}={} &\int_0^1[t^{2}(1-t)]^{k}\frac{\log (1-t)\D t}{1-t},\label{eq:(3k,k)c}\\\frac{1}{k\binom{4k}{2k}}={}&2\int_0^1[t^{}(1-t)]^{2k}\frac{\D t}{1-t},\label{eq:(4k,2k)a}\\
\frac{\mathsf H_{2k}-\mathsf H_{4k}}{k\binom{4k}{2k}}={}&2\int_0^1[t^{}(1-t)]^{2k}\frac{\log t\D t}{1-t},\label{eq:(4k,2k)b}\\
\frac{\mathsf H_{2k-1}-\mathsf H_{4k}}{k\binom{4k}{2k}}={}&2\int_0^1[t^{}(1-t)]^{2k}\frac{\log (1-t)\D t}{1-t},\label{eq:(4k,2k)c}
\end{align}These integral representations for binomial expressions will be the key to our analysis in the present work. Some of these integral formulae  have already been reported in the literature \cite{AKP2003,Au2020,Au2022a,Zhou2023SunCMZV,Sun2023}.

\subsection{Proofs of Theorems \ref{thm:3k_q}--\ref{thm:conj1.2}\label{subsec:Sun3kconj}}
We begin with a non-trivial application of \eqref{eq:(3k,k)a'}.\begin{proof}[Proof of Theorem \ref{thm:3k_q}]To prepare for the evaluation of an integral representation for  \eqref{eq:conj1.1}, we first establish \begin{align}\begin{split}&\frac{x-1 }{x^3}\int_0^1\frac{1 }{t^2 (1-t)}\log \left(1-\frac{ x^3t^2 (1-t)}{x-1}\right)\D t\\
={}&\frac{(1-x) q(x)^2}{x^3}-\frac{\sqrt{(1-x) (3+x)} q(x)}{x^2}-\frac{3 (1-x) \log ^2(1-x)}{4 x^3}-\frac{3 (1-x) \log (1-x)}{2 x^2}\end{split}\label{eq:Sun1.1}
\end{align}for $ x\in(-3,c)\smallsetminus\{0\}$.

Integrating by parts in $t$, we may identify the left-hand side of  \eqref{eq:Sun1.1} with\begin{align}
-\frac{(1-x) }{x^3}\int_0^1\left[ \frac{1}{t-\frac{1}{x}} +\frac{1}{t+\tau^+(x)}+\frac{1}{t+\tau^-(x)}\right]\left(\frac{1}{t}+\log \frac{1-t}{t}\right)\D t,
\end{align}where $ \tau^\pm(x)=\frac{1-x\pm i \sqrt{(1-x) (3+x)}}{2 x}$. Upon completing an elementary integration\begin{align}
\begin{split}
&-\frac{(1-x) }{x^3}\int_0^1\left[ \frac{1}{t-\frac{1}{x}} +\frac{1}{t+\tau^+(x)}+\frac{1}{t+\tau^-(x)}\right]\frac{\D t}{t}\\={}&-\frac{\sqrt{(1-x) (3+x)} q(x)}{x^2}-\frac{3 (1-x) \log (1-x)}{2 x^2},
\end{split}
\end{align} we are left with convergent integrals in the form of  \cite{Panzer2015}\begin{align}
\int_0^1\frac{1}{t+\tau}\log \frac{1-t}{t}\D t=\frac{\log ^2(1+\tau )}{2} +\frac{\log ^2\tau }{2}-\log (\tau ) \log (1+\tau ).\label{eq:3loglogs}
\end{align}After a little algebra with $ \tau\in\left\{-\frac1x,\tau^+(x),\tau^-(x)\right\}$, we arrive at the right-hand side of \eqref{eq:Sun1.1}.

Thanks to \eqref{eq:(3k,k)a'}, the left-hand side of \eqref{eq:conj1.1} is equal to\begin{align}
\begin{split}
\sum_{k=0}^\infty\frac{(3k+1)x^{3k}}{(k+1)(x-1)^{k}}\int_0^1[t(1-t^2)]^{2 k}\D t={}&\frac{x-1 }{x^3}\int_0^1\frac{2 }{t^2 (1-t)}\log \left(1-\frac{ x^3t^2 (1-t)}{x-1}\right)\D t\\{}&-\int_0^1\frac{3 (x-1)\D t}{1-x+t^2 (1-t) x^3}.\label{eq:conj1.1prep}
\end{split}
\end{align}Combining  \eqref{eq:Sun1.1} with the last elementary integral in \eqref{eq:conj1.1prep}, we obtain \eqref{eq:conj1.1} as claimed.\end{proof}\begin{remark}On several occasions in  this work, we use Panzer's \texttt{HyperInt} package \cite{Panzer2015}  to verify certain integral identities where the integrands involve logarithms and rational functions. For example, one may validate \eqref{eq:3loglogs} by typing \begin{quote}\texttt{fibrationBasis(fibrationBasis(hyperInt(log((1 - t)/t)/(t + tau), [t = 0 .. 1]), [tau]) - (log(1 + tau)\^{}2/2 + log(tau)\^{}2/2 - log(tau)*log(1 + tau)), [tau]);}\end{quote}in \texttt{Maple} and checking that the answer is $0$.
\eor\end{remark}

We  continue with an application of \eqref{eq:(3k,k)a}.

\begin{proof}[Proof of Theorem \ref{thm:conj1.2}] Exploiting \eqref{eq:(3k,k)a}, we may identify the left-hand side of \eqref{eq:conj1.2} with\begin{align}
\sum_{k=1}^\infty\f{x^{3k}}{k(x-1)^k\bi{3k}k}\int_0^1[t(1-t^2)]^{2 k}\frac{\D t}{1-t}=-\int_0^1\log \left(1-\frac{ x^3t^2 (1-t)}{x-1}\right)\frac{\D t}{1-t}.
\end{align}Integrating by parts,  one can recast the formula above into (cf.\ \cite{Panzer2015} for a  symbolic evaluation of the integral below) \begin{align}
\begin{split}{}&
-\int_0^1\left[ \frac{1}{t-\frac{1}{x}} +\frac{1}{t+\tau^+(x)}+\frac{1}{t+\tau^-(x)}\right]\log(1-t)\D t\\={}&-\frac{\log^2(1-x)+\log^2\big(1+\frac{1}{\tau^+(x)}\big)+\log^2\big(1+\frac{1}{\tau^-(x)}\big)}{2}\\&{}-\Li_2(x)-\Li_2\left(-\frac{1}{\tau^+(x)}\right)-\Li_2\left(-\frac{1}{\tau^-(x)}\right),
\end{split}
\end{align}where the dilogarithm function  is  analytically continued from  $ \Li_2(z)\colonequals\sum_{k=1}^\infty\frac{z^k}{k^2}$ for $ |z|\leq1$, being compatible with the MPL-GPL\ correspondence  $ \Li_2(z)\colonequals -G\big(0,\frac1z;1\big)$.

Since we have  $ \frac{1}{x}-{\tau^+}(x)-{\tau^-}(x)=1$, Newman's functional equation \cite[(1.44) or A.2.1(20)]{Lewin1981} fits here, in the form of\begin{align}
\begin{split}
&2\left[ \Li_2(x)+\Li_2\left(-\frac{1}{\tau^+(x)}\right)+\Li_2\left(-\frac{1}{\tau^-(x)}\right) \right]\\={}&\Li_2\left( -\frac{x\tau^-(x)}{\tau^+(x)} \right)+\Li_2\left( -\frac{x\tau^+(x)}{\tau^-(x)} \right)+\Li_2\left( -\frac{1}{x\tau^+(x)\tau^-(x)} \right).
\end{split}
\end{align} Owing to the elementary identities \begin{align}
-\frac{x\tau^-(x)}{\tau^+(x)}={}&\frac{1}{1+\tau^+(x)},&-\frac{x\tau^+(x)}{\tau^-(x)}={}&\frac{1}{1+\tau^-(x)},&-\frac{1}{x\tau^+(x)\tau^-(x)}=\frac{1}{1-\frac{1}{x}}{}&
\end{align} and the functional equations of Landen's type (cf.\ \cite[(1.12) or A.2.1(8)]{Lewin1981})\begin{align}
\Li_2(z)+\Li_2\left( \frac{1}{1-\frac{1}{z}} \right)=-\frac{\log^2(1-z)}{2},\quad z\in\left\{ x,-\frac{1}{\tau^+(x)},-\frac{1}{\tau^-(x)} \right\},
\end{align}we may further reduce Newman's relation into\begin{align}
\Li_2(x)+\Li_2\left(-\frac{1}{\tau^+(x)}\right)+\Li_2\left(-\frac{1}{\tau^-(x)}\right) =-\frac{\log^2(1-x)+\log^2\big(1+\frac{1}{\tau^+(x)}\big)+\log^2\big(1+\frac{1}{\tau^-(x)}\big)}{6},\label{eq:Newman_sum}
\end{align}  which gives the right-hand side of \eqref{eq:conj1.2}.
\end{proof}
\subsection{Proofs of Theorems \ref{thm:conj1.3}--\ref{thm:conj1.4}\label{subsec:4k2k}}
In this subsection, we will not need the full-fledged form for $ {4k \choose 2k}^{-1}$ given in \S\ref{subsec:3k4kintrepn}. Instead, for the moment, we only require a light version\begin{align}
\frac{1}{k\binom{2k}k}=\frac{1}{2}\int_0^1[t(1-t)]^{k-1}\D t
\end{align}that has  been covered in \cite[Corollary 3.4]{Zhou2022mkMpl}.
\begin{proof}[Proof of Theorem \ref{thm:conj1.3}]\begin{enumerate}[leftmargin=*,  label=(\alph*),ref=(\alph*),
widest=d, align=left] \item
Since \begin{align}
\sum_{k=0}^\infty\frac{1}{(k+1)X^{2k}\binom{4k}{2k}}=\sum_{k=0}^\infty\frac{1}{(k+2)X^{k}\binom{2k}{k}}+\sum_{k=0}^\infty\frac{1}{(k+2)(-X)^{k}\binom{2k}{k}}
\end{align}holds for  $ X>\frac{1}{4}$, it suffices to check that \begin{align}
\sum_{k=0}^\infty\frac{1}{(k+2)X^{k}\binom{2k}{k}}=\frac{4X(12X-1) \cot ^{-1}\sqrt{4X-1}}{\sqrt{4X-1}}-24X^2 \left( \cot ^{-1}\sqrt{4X-1} \right)^2-6X.\label{eq:2k_k_prep}
\end{align}The left-hand side of the equation above can be evaluated by\begin{align}
\frac{1}{2}+\sum_{k=1}^\infty\frac{k}{k+2}\int_0^1\left[\frac{t(1-t)}{X}\right]^k\frac{\D t}{2t(1-t)}=\frac{1}{2}+X\int_0^1\frac{ \frac{(t-1) t (t^2-t+2X)}{(t-1) t+X}-2X \log \frac{t^2-t+X}{X}}{2 (t-1)^3 t^3}\D t.
\end{align}Setting  $ u \colonequals\frac{1+i\sqrt{4X-1}}{2}$ so that $ \I u>0$ and  $ X=u(1-u)>\frac{1}{4}$, we may use Panzer's algorithm   \cite{Panzer2015} to identify the right-hand side of the last displayed equation with\begin{align}
-6 (1-u)u+\frac{2  (1-u)u [12 (1-u)u-1] \log\frac{u}{u-1}}{ 1-2u}+6 (1-u)^2 u^2 \log ^2\frac{u}{u-1},
\end{align}    which is the same as \eqref{eq:2k_k_prep}.
\item Writing $ X=\frac{i}{2}\sqrt{x(x-1)}$ in \eqref{eq:2k_k_prep},  performing analytic continuation to $\frac{X}{i}>\frac14$, while noting that\footnote{We compute the square roots of non-zero complex numbers by $\sqrt z\colonequals\sqrt{|z|}e^{\frac{i}{2}\arg z} $, where $-\pi<\arg z\leq\pi $.} {\allowdisplaybreaks\begin{align}
\frac{1}{\sqrt{4X-1}}={}&\frac{\sqrt{-x}-\sqrt{x-1}}{1-2x},\\\frac{1}{\sqrt{4X+1}}={}&\frac{\sqrt{1-x}-\sqrt{x}}{1-2 x},\\ \tanh ^{-1}\frac{1}{\sqrt{4X+1}}+i\cot ^{-1}\sqrt{4X-1}={}&\begin{cases}\tanh^{-1}\frac{1}{\sqrt{x}}, &\text{if } x>\frac{1+\sqrt{2}}{2} ,\\
\tanh^{-1}\frac{1}{\sqrt{1-x}}, & \text{if }x<\frac{1-\sqrt{2}}{2}, \\
\end{cases}\\\tanh ^{-1}\frac{1}{\sqrt{4X+1}}-i\cot ^{-1}\sqrt{4X-1}={}&\begin{cases}\tanh^{-1}\frac{1}{\sqrt{1-x}}, & \text{if }x>\frac{1+\sqrt{2}}{2} ,\\
\tanh^{-1}\frac{1}{\sqrt{x}}, & \text{if }x<\frac{1-\sqrt{2}}{2}, \\
\end{cases}
\end{align}}we immediately reach the claimed identity.\qedhere\end{enumerate}\end{proof} \begin{proof}[Proof of Theorem \ref{thm:conj1.4}] For $X>1/4$, one can establish\begin{align}\sum_{k=1}^\infty\f1{k^2X^{2k}\bi{4k}{2k}}=4\left(\cot^{-1}\sqrt{4X-1}\right)^{2}-\log^2\f{\sqrt{4X+1}+1}{\sqrt{4X+1}-1}\end{align}
on the observations that \begin{align}
\f12\sum_{k=1}^\infty\f1{k^2X^{2k}\bi{4k}{2k}}=\sum_{k=1}^\infty\f{1+(-1)^k}{k^2X^k\bi{2k}k}
\end{align}and \begin{align}
\sum_{k=1}^\infty\f {z^{2k}}{k^2\bi{2k}{k}}=2\left(\sin^{-1}\frac{z}{2}\right)^{2}=\begin{cases}2\left(\cot^{-1}\sqrt{4X-1}\right)^2 &\text{if } z=1/\sqrt{X}\in(0,2),\ \\
-\frac{1}{2} \log^2\f{\sqrt{4X+1}+1}{\sqrt{4X+1}-1}& \text{if } z/i=1/\sqrt{X}\in(0,2). \\
\end{cases}
\end{align} Again, the substitution $ X=\frac{i}{2}\sqrt{x(x-1)}$ will bring us a new summation formula, as declared in \eqref{eq:4k2ksymm}.\end{proof}

\section{Polylogarithmic reductions of certain inverse binomial sums\label{sec:polylog_red}}
\subsection{Polylogarithmic characterizations of certain infinite series\label{subsec:GPLchar}}As extensions to  the proof of \cite[Theorem 1.3(a)]{Zhou2023SunCMZV}, the arguments in the current subsection draw on repeated invocations of the recursive structure for GPLs, as given in \eqref{eq:GPL_rec}.\begin{proof}[Proof of Theorem \ref{thm:3kGPL}]\begin{enumerate}[leftmargin=*,  label=(\alph*),ref=(\alph*),
widest=d, align=left] \item
For  $r\in\mathbb Z_{>0} $,  we first show inductively that\begin{align}
\Li_r\left(\frac{x^{3}t^{2}(1-t)}{x-1} \right)\in \mathfrak{G}^{(t)}_{r;r}\left[\varnothing;\left\{ \frac{1}{x},-\tau^+ (x),-\tau^-(x)\right\} \right].\label{eq:Li_fib_3k}
\end{align} When $r=1$, the statement is evident from the relation\begin{align}
\Li_1\left(\frac{x^{3}t^{2}(1-t)}{x-1} \right)=-G\left( \frac{1}{x} ;t\right)-G(-\tau^+(x);t)-G(-\tau^-
(x);t).\label{eq:Li1_GGG}\end{align}Assuming that \eqref{eq:Li_fib_3k} is true for $ r\in\mathbb Z_{>0}$, we may  verify the same claim for $r+1$ by applying the   GPL recursion \eqref{eq:GPL_rec}  to the  following relation:\begin{align}
\Li_{r+1}\left(\frac{x^{3}t^{2}(1-t)}{x-1} \right)=\int_{0}^1\left( \frac{2}{s} +\frac{1}{s-1}\right)\Li_{r}\left(\frac{x^{3}s^{2}(1-s)}{x-1} \right)\D s.\label{eq:Li_rec}
\end{align}So far, we see that   \eqref{eq:GPL_rec}, \eqref{eq:Li_fib_3k}, and\begin{align}
\sum_{k=1}^\infty\frac{1}{k^{r+2}\binom{3k}k}\left( \frac{x^{3}}{x-1} \right)^{k}=\int_0^1\Li_{r+1}\left(\frac{x^{3}t^{2}(1-t)}{x-1} \right)\frac{\D t}{1-t},\quad r\in\mathbb Z_{>0}\label{eq:3k_zeta}
\end{align}  lead us  to  \eqref{eq:3k_inv_sum1} for $a_k=\frac1k ,k\in\mathbb Z_{>0}$. In view of \eqref{eq:Li1_GGG} and \eqref{eq:3k_zeta},  we may extend the validity of   \eqref{eq:3k_inv_sum1} back to $r=0$ for $x\in(-3,c)$ and  $a_k=\frac1k ,k\in\mathbb Z_{>0}$.

To move onto the cases where $ (a_k)_{k\in\mathbb Z_{>0}}\in\big\{(\mathsf H_{k}-\mathsf H_{3k})_{k\in\mathbb Z_{>0}},$ $(\mathsf H_{2k}-\mathsf H_{3k})_{k\in\mathbb Z_{>0}}\big\}$, we recall from \eqref{eq:GPL_rec}  that $ G(1;t)=\log(1-t)$ and $G(0;t)=\log t$, before employing  a shuffle product    \cite[(2.4)]{Frellesvig2016}\begin{align}
\begin{split}
G(\alpha;t)G(\beta_1,\dots,\beta_r;t)={}&G(\alpha,\beta_1,\dots,\beta_r;t)+\sum_{j=1}^{r-1}G(\beta_{1},\dots,\beta_j,\alpha,\beta_{j+1},\dots ,\beta_r;t)\\{}&+G(\beta_1,\dots,\beta_r,\alpha;t)\label{eq:ShGPL}
\end{split}
\end{align}to deduce\begin{align}
\Li_{r}\left(\frac{x^{3}t^{2}(1-t)}{x-1} \right)\log t\in {}&\mathfrak{G}^{(t)}_{r+1;r}\left[\varnothing;\left\{ \frac{1}{x},-\tau^+ (x),-\tau^-(x)\right\} \right],\\\Li_{r}\left(\frac{x^{3}t^{2}(1-t)}{x-1} \right)\log (1-t)\in {}&\mathfrak{G}^{(t)}_{r+1;r}\left[\varnothing;\left\{ \frac{1}{x},-\tau^+ (x),-\tau^-(x)\right\} \right].
\end{align}Subsequent invocations of the  GPL recursion \eqref{eq:GPL_rec} bring us   \eqref{eq:3k_inv_sum1} for $a_k=\mathsf H_{k}-\mathsf H_{3k} =\mathsf H_{k-1}-\mathsf H_{3k} +\frac{1}{k},k\in\mathbb Z_{>0}$ and  $a_k=\mathsf H_{2k}-\mathsf H_{3k} ,k\in\mathbb Z_{>0}$. Here, the requirement  $r\in\mathbb Z_{>0}$ cannot be relaxed, as the situation $r=0$ corresponds to different structures [see \eqref{eq:H2k-H3k_spec} and \eqref{eq:Hk-H3k_spec}  below].

From the proof of \cite[Theorem 3.1]{Zhou2022mkMpl}, we know that \begin{align}
\sum_{k=1}^\infty\frac{\prod_{j=1}^M\mathsf H_{k}^{(r_{j})}}{k^{r}}z^k\in\mathfrak G_{w;w}^{(z)}[\varnothing;\{1\}]\label{eq:H_prod_gf}
\end{align}for $ |z|<1$ and $ r\in\mathbb Z_{>0}$, where $w=r+\sum_{j=1}^Mr_j$. Now, in parallel to \eqref{eq:Li_rec}, we have{\allowdisplaybreaks\begin{align}
\begin{split}&
G\left(0,\alpha_{1},\dots,\alpha_r;\frac{x^{3}t^{2}(1-t)}{x-1} \right)\\={}&\int_{0}^1\left( \frac{2}{s} +\frac{1}{s-1}\right)G\left(\alpha_{1},\dots,\alpha_r;\frac{x^{3}s^{2}(1-s)}{x-1} \right)\D s,
\end{split}\\\begin{split}&
G\left(1,\alpha_{1},\dots,\alpha_r;\frac{x^{3}t^{2}(1-t)}{x-1} \right)\\={}&\int_{0}^1\left[ \frac{1}{s-\frac{1}{x}} +\frac{1}{s+\tau ^+(x)}+\frac{1}{s+\tau ^-(x)}\right]G\left(\alpha_{1},\dots,\alpha_r;\frac{x^{3}s^{2}(1-s)}{x-1} \right)\D s,
\end{split}
\end{align}}by virtue of the  GPL recursion \eqref{eq:GPL_rec} (up to a variable substitution), so \eqref{eq:H_prod_gf} entails\begin{align}
\sum_{k=1}^\infty\frac{\prod_{j=1}^M\mathsf H_{k}^{(r_{j})}}{k^{r}}\left[\frac{x^{3}t^{2}(1-t)}{x-1}\right]^k\in\mathfrak G_{w;w}^{(t)}\left[\left\{ \frac{1}{x},-\tau^+ (x),-\tau^-(x)\right\};\left\{ \frac{1}{x},-\tau^+ (x),-\tau^-(x)\right\} \right],\label{eq:Hprod_gf3k_prep}
\end{align}hence \eqref{eq:H_prod_sum}.
\item We  note that the closed forms for \begin{align}
\sum_{k=0}^\infty\frac{1}{(3k+1)\binom{3k}k}\left( \frac{x^{3}}{x-1} \right)^{k},\quad \sum_{k=0}^\infty\frac{1}{(3k+2)\binom{3k}k}\left( \frac{x^{3}}{x-1} \right)^{k},\quad \sum_{k=1}^\infty\frac{1}{(2k-1)\binom{3k}k}\left( \frac{x^{3}}{x-1} \right)^{k}\label{eq:Sun3k_sums}
\end{align}were the subject of  \cite[Theorem 1.1]{Sun2023}. What we will do soon is  a  modest generalization of  \cite[Theorems 1.1 and  1.3]{Sun2023}.

Thanks to \eqref{eq:(3k,k)b'} and \eqref{eq:(3k,k)c'}, we may represent the left-hand side of \eqref{eq:Hk-H2k(3k+1)} as an integral\begin{align}
\int_0^1\frac{\log\frac{1-t}{t}\D t}{1-\frac{ x^3t^2 (1-t)}{x-1}}.
\end{align}Before simplifying this integral with the help from \eqref{eq:3loglogs}, we  need a partial fraction expansion\begin{align}
\begin{split}
\frac{1}{ 1-\frac{ x^3t^2 (1-t)}{x-1}}={}&-\frac{1-x}{x (3-2 x)}\frac{1}{t-\frac{1}{x}}+\left[\frac{1-x}{2x (3-2 x)}+i\sqrt{\frac{1-x}{3+x}}\frac{3-x}{2 x (3-2 x)}\right]\frac{1}{t+\tau^+(x)}\\{}&+\left[\frac{1-x}{2x (3-2 x)}-i\sqrt{\frac{1-x}{3+x}}\frac{3-x}{2 x (3-2 x)}\right]\frac{1}{t+\tau^-(x)}.
\end{split}
\end{align}It is then an elementary exercise to check that\begin{align}
\int_0^1\left[ -\frac{2}{t-\frac{1}{x}}+ \frac{1}{t+\tau^+(x)}+\frac{1}{t+\tau^-(x)}\right]\log\frac{1-t}{t}\D t={}&-q(x)^2-\frac{3 \log ^2(1-x)}{4}
\intertext{and}
\int_0^1\left[  \frac{1}{t+\tau^+(x)}-\frac{1}{t+\tau^-(x)}\right]\log\frac{1-t}{t}\D t={}&i q(x) \log (1-x),
\end{align}so  the right-hand side of \eqref{eq:Hk-H2k(3k+1)} immediately follows.

In a similar vein as the last paragraph, one can identify the left-hand side of \eqref{eq:Hk-H3k(3k+1)} with  \begin{align}
\int_0^1\frac{\log(1-t)\D t}{1-\frac{ x^3t^2 (1-t)}{x-1}},
\end{align} and subsequently compute [cf.\ \eqref{eq:Newman_sum}]\begin{align}
\begin{split}
&\int_0^1\left[ -\frac{2}{t-\frac{1}{x}}+ \frac{1}{t+\tau^+(x)}+\frac{1}{t+\tau^-(x)}\right]\log(1-t)\D t\\={}&-3 \Li_2(x)-\frac{2q(x)^2}{3}-\log ^2(1-x),
\end{split} \\
\begin{split}
&\int_0^1\left[  \frac{1}{t+\tau^+(x)}-\frac{1}{t+\tau^-(x)}\right]\log(1-t)\D t\\={}&\Li_2\left( -\frac{1}{\tau^+(x)} \right)-\Li_2\left( -\frac{1}{\tau^-(x)} \right)+i q(x) \log (1-x),
\end{split}
\end{align}hence the right-hand side of \eqref{eq:Hk-H3k(3k+1)}.

We have the following integral formulae by variations on \eqref{eq:(3k,k)a}--\eqref{eq:(3k,k)c}:\begin{align}
\left(\frac{1}{3k+1}+\frac{1}{3k+2}\right)\frac{1}{\binom{3k}k}={}&\frac{9}{2}\int_0^1[t^2(1-t)]^{k+1}\frac{\D t}{1-t},\label{eq:3k+2_int}\\\left(\frac{1}{3k+1}+\frac{1}{3k+2}\right)\frac{\mathsf H_{2 k+2}-\mathsf H_{3 k+3}}{\binom{3k}k}={}&\frac{9}{2}\int_0^1[t^2(1-t)]^{k+1}\frac{\log t\D t}{1-t},\label{eq:3k+2_int'}\\\left(\frac{1}{3k+1}+\frac{1}{3k+2}\right)\frac{\mathsf H_{k}-\mathsf H_{3 k+3}}{\binom{3k}k}={}&\frac{9}{2}\int_0^1[t^2(1-t)]^{k+1}\frac{\log (1-t)\D t}{1-t},\label{eq:3k+2_int''}
\end{align}which enable us to compute\begin{align}
\begin{split}&
\sum_{k=0}^\infty\left(\frac{1}{3k+1}+\frac{1}{3k+2}\right)\frac{\mathsf H_{k}-\mathsf H_{2k+2}}{\binom{3k}k}\left( \frac{x^{3}}{x-1} \right)^{k}=\frac{9}{2}\int_0^1\frac{t^2(1-t)}{1-\frac{x^{3}}{x-1}t^2(1-t)}\frac{\log \frac{1-t}{t}\D t}{1-t}\\={}&-\frac{9(1-x) }{4x^3 (3-2 x)} \log ^2(1-x)+\frac{9 (1-x)^2}{8 x^3 (3-2 x)} \left[4q(x)^2-\log ^2(1-x)\right]\\{}&+\frac{ 9(1-x)}{2x^2 (3-2 x)} \sqrt{\frac{1-x}{3+x}} q(x) \log (1-x)
\end{split}
\end{align}and{\allowdisplaybreaks\begin{align}
\begin{split}&
\sum_{k=0}^\infty\left(\frac{1}{3k+1}+\frac{1}{3k+2}\right)\frac{\mathsf H_{2 k+2}-\mathsf H_{2k}}{\binom{3k}k}\left( \frac{x^{3}}{x-1} \right)^{k}\\={}&\sum_{k=0}^\infty\left[ \frac{15}{4 (3 k+1)}-\frac{3}{2 (3 k+2)}-\frac{3}{4 (k+1)} \right]\frac{1}{\binom{3k}{k}}\left( \frac{x^{3}}{x-1} \right)^{k}\\={}&\frac34\int_{0}^1\frac{(7-9t^{2})\D t}{1-\frac{t^2 (1-t) x^3}{x-1}}-\frac{3}{4}\sum_{k=0}^\infty\frac{1}{(k+1)\binom{3k}{k}}\left( \frac{x^{3}}{x-1} \right)^{k}\\={}&-\frac{3\left(7 x^2-3 x-18\right)}{4x^2 (3-2 x)}\sqrt{\frac{1-x}{3+x}}q(x)+\frac{9 (1-x) (6-7 x) }{8 x^2 (3-2 x)}\log (1-x)\\{}&-\frac{3}{4}\sum_{k=0}^\infty\frac{1}{(k+1)\binom{3k}{k}}\left( \frac{x^{3}}{x-1} \right)^{k}.\label{eq:3k+2patch}
\end{split}
\end{align}}Here, the last infinite sum has been treated in Theorem \ref{thm:3k_q}.
 These
efforts are then combined into \eqref{eq:Hk-H2k(3k+2)}.

  One can  reformulate the last sum in \eqref{eq:Sun3k_sums} as\begin{align}
\sum_{k=1}^\infty\frac{1}{(2k-1)\binom{3k}k}\left( \frac{x^{3}}{x-1} \right)^{k}=\frac{2}{3}\int_0^1\frac{\frac{x^{3}(1-t)}{x-1}\D t}{1-\frac{x^{3}t^2(1-t)}{x-1}}.
\end{align}By extension of the procedures in \S\ref{subsec:3k4kintrepn}, we may evaluate \eqref{eq:Hk-H2k(2k-1)} and \eqref{eq:Hk-H3k(2k-1)}  via the following integral representations:\begin{align}\frac{\mathsf H_{2k-2}-\mathsf H_{3 k-1}}{(2k-1)\binom{3k}k}={}&\frac{2}{3}\int_0^1[t^2(1-t)]^k\frac{\log t}{t^2}\D t,\\
\frac{\mathsf H_k-\mathsf H_{3 k-1}}{(2k-1)\binom{3k}k}={}&\frac{2}{3}\int_0^1[t^2(1-t)]^k\frac{\log (1-t)}{t^2}\D t.
\end{align} We leave the details to diligent readers, since they are essentially similar to the foregoing proof of  \eqref{eq:Hk-H2k(3k+1)} and \eqref{eq:Hk-H3k(3k+1)}.
 \qedhere\end{enumerate}
\end{proof}\begin{remark}We may supplement part (a) of Theorem \ref{thm:3kGPL}  by taking derivatives in the following manner:\begin{align}
\sum_{k=1}^\infty\frac{a_{k}}{k^{2-n}\binom{3k}k}\left( \frac{x^{3}}{x-1} \right)^{k}=\left[\frac{x (1-x)}{3-2 x} \frac{\D}{\D x}\right]^n\sum_{k=1}^\infty\frac{a_{k}}{k^{2}\binom{3k}k}\left( \frac{x^{3}}{x-1} \right)^{k},
\end{align}where $x\in(-3,c) $ and $ n\in\mathbb Z_{>0}$. In particular, one has {\allowdisplaybreaks\begin{align}\begin{split}
\sum_{k=1}^\infty\frac{1}{k^{}\binom{3k}k}\left( \frac{x^{3}}{x-1} \right)^{k}={}&  \frac{2x}{3-2 x}\sqrt{\frac{1-x}{3+x}}q(x)+\frac{x \log (1-x)}{3-2 x},\label{eq:2022ak+b}
\end{split}\\\begin{split}
\sum_{k=1}^\infty\frac{\mathsf H_{2k}-\mathsf H_{3k}}{k^{}\binom{3k}k}\left( \frac{x^{3}}{x-1} \right)^{k}={}&\frac{x \Li_2(x)}{3-2 x}+\frac{i x   }{3-2 x}\sqrt{\frac{1-x}{3+x}}\left[\Li_2\left(-\frac{1}{\tau ^+(x)}\right)-\Li_2\left(-\frac{1}{\tau ^-(x)}\right)\right]\\{}&+\frac{1-x}{3-2x}\left[\frac{q(x)^2}{3}-\frac{\log ^2(1-x)}{4}  \right],\label{eq:H2k-H3k_spec}
\end{split}\\\begin{split}
\sum_{k=1}^\infty\frac{\mathsf H_{k-1}-\mathsf H_{2k}}{k^{}\binom{3k}k}\left( \frac{x^{3}}{x-1} \right)^{k}={}&-\frac{(1-x)q(x)^2}{3-2 x}-\frac{x }{3-2 x} \sqrt{\frac{1-x}{3+x}}q(x) \log (1-x)\\{}&+\frac{(3-x) \log ^2(1-x)}{4 (3-2 x)}.\label{eq:Hk-H3k_spec}
\end{split}
\end{align}}Here, series like   \eqref{eq:2022ak+b} were first investigated in \cite{Sun2022ab}.        \eor\end{remark}\begin{remark}From \cite[(25)]{Mezo2014}, one can read off a particular case of \eqref{eq:H_prod_gf}: $ \sum_{k=1}^\infty\frac{\mathsf H_k}{k}z^k=\Li_2(z)+\frac{\log^2(1-z)}{2}$ for $ |z|<1$. This formula will become useful for certain tabulated entries in  \S\ref{subsec:CMZVchar}.
\eor\end{remark}

As seen from \S\ref{subsec:4k2k} and \cite[Remark 1.6]{Zhou2023SunCMZV}, one may fall back on \cite[Corollary 3.4]{Zhou2022mkMpl} (see also \cite{KalmykovVeretin2000,DavydychevKalmykov2001,DavydychevKalmykov2004,Weinzierl2004bn,Kalmykov2007,ABRS2014,Ablinger2017}) for infinite series involving $ \binom{4k}{2k}^{-1}$, $ \mathsf H_{2k}$, and $ \mathsf H_{4k}$. Instead of dwelling on such straightforward consequences of infinite series involving $ \binom{2k}{k}^{-1}$, $ \mathsf H_{k}$, and $ \mathsf H_{2k}$, we will focus on scenarios where the integral representations in \S\ref{subsec:3k4kintrepn} become indispensable, namely the cases treated in Theorem \ref{thm:4kGPL}.
\begin{proof}[Proof of Theorem \ref{thm:4kGPL}]\begin{enumerate}[leftmargin=*,  label=(\alph*),ref=(\alph*),
widest=d, align=left] \item
With the help from \eqref{eq:(4k,2k)a}, we may rewrite the left-hand side of \eqref{eq:H_prod_sum4k} as\begin{align}
\int_0^1\sum_{k=1}^\infty\frac{\prod_{j=1}^M\mathsf H_{k}^{(r_{j})}}{k^{r}}\left[ \frac{2^{2}t^2(1-t)^2}{x(1-x)} \right]^{k}\frac{\D t}{1-t}.
\end{align}    A routine variation on the proof of \eqref{eq:Hprod_gf3k_prep} brings us \begin{align}
\begin{split}
&\sum_{k=1}^\infty\frac{\prod_{j=1}^M\mathsf H_{k}^{(r_{j})}}{k^{r}}\left[ \frac{2^{2}t^2(1-t)^2}{x(1-x)} \right]^{k}\\\in{}&{\mathfrak{G}}^{(1)}_{w;w}\left[\left\{ \xi_{0,0}(x),\xi_{0,1}(x),\xi_{1,0}(x),\xi_{1,1}(x)\right\};\left\{ \xi_{0,0}(x),\xi_{0,1}(x),\xi_{1,0}(x),\xi_{1,1}(x)\right\} \right],
\end{split}
\end{align}where  $w=r+\sum_{j=1}^Mr_j$, so the right-hand side of  \eqref{eq:H_prod_sum4k} follows from an application of the GPL recursion \eqref{eq:GPL_rec}.
\item In view of \eqref{eq:(4k,2k)a'}, we may evaluate the left-hand side of \eqref{eq:Hk(4k+1)} via\begin{align}
\int_0^1\frac{\Li_1\left( \frac{[2t (1-t)]^2}{x (1-x)} \right)\D t}{1-\frac{[2t (1-t)]^2}{x (1-x)}},
\end{align}where\begin{align}
\Li_1\left( \frac{[2t (1-t)]^2}{x (1-x)} \right)=-\sum_{\ell,m\in\{0,1\}}\log\left(1- \frac{t}{\xi_{\ell,m}(x)} \right)=-\sum_{\ell,m\in\{0,1\}} G(\xi_{\ell,m}(x);t)
\end{align}and\begin{align}
\frac{1}{1-\frac{[2t (1-t)]^2}{x (1-x)}}=\frac{\sqrt{(1-x) x}}{4 (2 x-1)}\sum_{\ell,m\in\{0,1\}}\frac{(-1)^{\ell}\sqrt{1-x}-(-1)^m\sqrt{x}}{t-\xi_{\ell,m}(x)}.
\end{align}Thus, the standard GPL recursion brings us the right-hand side of \eqref{eq:Hk(4k+1)}.

Analogously, the integral\begin{align}
\int_0^1\frac{\Li_1\left( \frac{[2t (1-t)]^2}{x (1-x)} \right)}{1-\frac{[2t (1-t)]^2}{x (1-x)}}[16 t^2 (1-t)-1]\D t
\end{align}leads to a GPL representation of \eqref{eq:Hk(4k+3)}.\qedhere\end{enumerate}
\end{proof} \begin{table}[b]\caption{Selected CMZV characterizations of  $\mathsf S_{3,r}\left( a_k;z\right)$ at level $5$, where $ \phi\colonequals\frac{1+\sqrt{5}}{2}$, $ \varsigma\colonequals e^{2\pi i/5}$, and $ \pd\colonequals\log \phi$\label{tab:Zk(5)3k}}

\begin{scriptsize}\begin{align*}\begin{array}{@{}c@{}l|@{}c@{}l} \hline\hline
\vphantom{\frac{\int}{1}}&\mathsf S_{3,0}\left(\frac1k;\mathfrak r_{1/5}=\phi\right)=\int_0^1\Li_1\left( \phi t ^2(1-t)\right)\frac{\D t}{1-t}&&\mathsf S_{3,0}\left(\frac1k;\mathfrak r_{2/5}=-\frac{1}{\phi}\right)=\int_0^1\Li_1\left( -\frac{t ^2(1-t)}{\phi}\right)\frac{\D t}{1-t}
\\
{}={}&\frac{8 \pi ^2}{75}-2 \pd ^2&{}={}&\frac{2 \pi ^2}{75}-2 \pd ^2 \\[5pt]&\mathsf S_{3,1}\left(\frac1k;\phi\right)=\int_0^1\Li_2\left( \phi t ^2(1-t)\right)\frac{\D t}{1-t}&&\mathsf S_{3,1}\left(\frac1k;-\frac{1}{\phi}\right)=\int_0^1\Li_2\left( -\frac{t ^2(1-t)}{\phi}\right)\frac{\D t}{1-t}\\{}={}&-\frac{5 \RE\Li_3(\varsigma )}{2}-\frac{84 \zeta (3)}{25}+\frac{4\pi  \I\left[2\Li_2\big(\varsigma ^2\big)+\Li_2(\varsigma )\right]}{5}     &{}={}&\frac{5 \RE\Li_3(\varsigma )}{2}-\frac{54 \zeta (3)}{25}-\frac{2\pi\I\left[\Li_2\big(\varsigma ^2\big)-2\Li_2(\varsigma )\right]  }{5}   \\&{} -\frac{2 \pd ^3}{3}+\frac{8 \pi ^2 \pd }{75}&&{}+\frac{2 \pd ^3}{3}-\frac{2 \pi ^2 \pd }{75}\\[3pt]\hline \vphantom{\frac{\int}{1}}&\mathsf S_{3,1}\left(\mathsf H_{2k}-\mathsf H_{3k};\phi\right)=\int_0^1\Li_1\left( \phi t ^2(1-t)\right)\frac{\log t\D t}{1-t}&\vphantom{\frac{\int}{1}}&\mathsf S_{3,1}\left(\mathsf H_{2k}-\mathsf H_{3k};-\frac{1}{\phi}\right)=\int_0^1\Li_1\left( -\frac{t ^2(1-t)}{\phi}\right)\frac{\log t\D t}{1-t}\\{}={}&-\frac{25 \RE\Li_3(\varsigma )}{12}-\frac{14 \zeta (3)}{5}+\frac{2\pi  \I\left[2\Li_2\big(\varsigma ^2\big)+\Li_2(\varsigma )\right]}{5}    &{}={}&\frac{25 \RE\Li_3(\varsigma )}{12}-\frac{9 \zeta (3)}{5}-\frac{\pi\I\left[\Li_2\big(\varsigma ^2\big)-2\Li_2(\varsigma )\right]  }{5}   \\&{} +\frac{\pd ^3}{3}+\frac{19 \pi ^2 \pd }{75}&&{}-\frac{\pd ^3}{3}+\frac{14 \pi ^2 \pd }{75}\\[3pt]\hline \vphantom{\frac{\int}{1}}&\mathsf S_{3,1}\left(\mathsf H_{k-1}-\mathsf H_{3k};\phi\right)=\int_0^1\Li_1\left( \phi t ^2(1-t)\right)\frac{\log(1-t)\D t}{1-t}&\vphantom{\frac{\int}{1}}&\mathsf S_{3,1}\left(\mathsf H_{k-1}-\mathsf H_{3k};-\frac{1}{\phi}\right)=\int_0^1\Li_1\left( -\frac{t ^2(1-t)}{\phi}\right)\frac{\log(1- t)\D t}{1-t}\\{}={}&-\frac{5 \RE\Li_3(\varsigma )}{6}-\frac{28 \zeta (3)}{25}+\frac{7 \pi ^2 \pd }{75} &{}={}&\frac{5 \RE\Li_3(\varsigma )}{6}-\frac{18 \zeta (3)}{25}+\frac{17 \pi ^2 \pd }{75}\\[3pt]\hline\hline
\end{array}\end{align*}\end{scriptsize}\end{table}
\begin{table}[t]\caption{Selected CMZV characterizations of $ \mathsf S_{3,r}\left(a_k;z\right)$ at level $6$, where $ \omega\colonequals e^{2\pi i/3}$, $ \varrho\colonequals e^{\pi i/3}$, $ \lambda\colonequals\log2$,  and $ \varLambda\colonequals\log3$\label{tab:Zk(6)3k}}

\begin{scriptsize}\begin{align*}\begin{array}{@{}c@{}l|@{}c@{}l} \hline\hline
\vphantom{\frac{\int}{1}}&\mathsf S_{3,0}\left(\frac1k;\mathfrak r_{1/6}=\frac{2^3}{3}\right)=\int_0^1\Li_1\left( \frac{2^{3}}{3}t ^2(1-t)\right)\frac{\D t}{1-t}&&\mathsf S_{3,0}\left(\frac1k;\mathfrak r_{0}=\frac{3^3}{2^2}\right)=\int_0^1\Li_1\left( \frac{3^3}{2^2}t ^2(1-t)\right)\frac{\D t}{1-t}\\{}={}&\frac{\pi^2}{6}-\frac{\varLambda^2}{2}&{}={}&\frac{2 \pi ^2}{3}-2 \lambda ^2\\[5pt]&\mathsf S_{3,1}\left(\frac1k;\frac{2^3}{3}\right)=\int_0^1\Li_2\left( \frac{2^{3}}{3}t ^2(1-t)\right)\frac{\D t}{1-t}&&\mathsf S_{3,1}\left(\frac1k;\frac{3^3}{2^2}\right)=\int_0^1\Li_2\left( \frac{3^3}{2^2}t ^2(1-t)\right)\frac{\D t}{1-t}\\{}={}&12 \RE\Li_{1,1,1}(\omega ,1,\varrho )+\frac{\zeta (3)}{6}+6 \varLambda  \RE\Li_{1,1}(\omega ,\varrho )   &{}={}&24 \RE\Li_{1,1,1}(-1,1,\varrho )-\frac{22 \zeta (3)}{3} -24 \lambda  \RE\Li_{1,1}(\omega ,\varrho )\\&{} -\frac{\pi  \I\Li_2(\omega )}{2}-\frac{\varLambda ^2 (9 \lambda -\varLambda )}{6} +\frac{\pi ^2 (3 \lambda -\varLambda )}{6} &&{} +\frac{2\lambda ^2 (2 \lambda +9 \varLambda )}{3} -\frac{2\pi ^2 (2 \lambda -3 \varLambda )}{3}  \\[3pt]
\hline \vphantom{\frac{\int}{1}}&\mathsf S_{3,1}\left(\mathsf H_{2k}-\mathsf H_{3k};\frac{2^3}{3}\right)=\int_0^1\Li_1\left( \frac{2^{3}}{3}t ^2(1-t)\right)\frac{\log t\D t}{1-t}&&\mathsf S_{3,1}\left(\mathsf H_{2k}-\mathsf H_{3k};\frac{3^3}{2^2}\right)=\int_0^1\Li_1\left( \frac{3^3}{2^2}t ^2(1-t)\right)\frac{\log t\D t}{1-t}\\{}={}&10 \RE\Li_{1,1,1}(\omega ,1,\varrho ) +\frac{5 \zeta (3)}{36} +3 \varLambda  \RE\Li_{1,1}(\omega ,\varrho )   &{}={}&20 \RE\Li_{1,1,1}(-1,1,\varrho )-\frac{55 \zeta (3)}{9}-12 \lambda  \RE\Li_{1,1}(\omega ,\varrho )\\&{}   -\frac{5\pi  \I\Li_2(\omega )}{4}-\frac{\lambda  \varLambda ^2}{4}+\frac{\pi ^2 (3 \lambda +4 \varLambda )}{36} &&{} +5 \lambda ^2 \varLambda +\frac{\pi ^2 (2 \lambda +3 \varLambda )}{9}  \\[3pt]\hline \vphantom{\frac{\int}{1}}&\mathsf S_{3,1}\left(\mathsf H_{k-1}-\mathsf H_{3k};\frac{2^3}{3}\right)=\int_0^1\Li_1\left( \frac{2^{3}}{3}t ^2(1-t)\right)\frac{\log (1-t)\D t}{1-t}&&\mathsf S_{3,1}\left(\mathsf H_{k-1}-\mathsf H_{3k};\frac{3^3}{2^2}\right)=\int_0^1\Li_1\left( \frac{3^3}{2^2}t ^2(1-t)\right)\frac{\log(1- t)\D t}{1-t}\\{}={}&4 \RE\Li_{1,1,1}(\omega ,1,\varrho )+\frac{\zeta (3)}{18}-\pi  \I\Li_2(\omega )&{}={}&8 \RE\Li_{1,1,1}(-1,1,\varrho )-\frac{22 \zeta (3)}{9} \\&{}+ \frac{\varLambda ^2 (3 \lambda -\varLambda )}{6} -\frac{\pi ^2 (3 \lambda -2 \varLambda )}{18}  &&{}-\frac{2 \lambda ^2 (2 \lambda -3 \varLambda )}{3} +\frac{2\pi ^2 (\lambda -3 \varLambda )}{9} \\[3pt]\hline &\text{CMZVs  unavailable from Au's test for }\mathsf S_{3,1}\left(\mathsf H_{k};\frac{2^3}{3}\right)&&\vphantom{\frac{\frac\int\int}{\int}}\mathsf S_{3,1}\left(\mathsf H_{k};\frac{3^3}{2^2}\right)=\int_0^1\left[\Li_2\left( \frac{3^3}{2^2}t ^2(1-t)\right)+\frac{\log^2\left( 1-\frac{3^3}{2^2}t ^2(1-t) \right)}{2}\right]\frac{\D t}{1-t}\\&&{}={}&72 \RE\Li_{1,1,1}(-1,1,\varrho )+4 \zeta (3)-48 \lambda  \RE\Li_{1,1}(\omega ,\varrho )\\&&&{} +18 \lambda ^2 \varLambda -2 \pi ^2 (\lambda -\varLambda )\\[1pt]\hline\hline
\end{array}\end{align*}\end{scriptsize}\end{table}\subsection{CMZV representations of certain infinite series\label{subsec:CMZVchar}}
In principle, for certain sets $S,S_*\subset\mathbb C $, we can embed $\mathfrak G_{k,k_*}^{(1)}[S,S_*] $ into the CMZV space $ \mathfrak Z_k(N)$ [defined in \eqref{eq:Zk(N)_defn}] for  a suitable $N\in\mathbb Z_{>0}$, by fibrations of GPLs \cite[Lemma 2.14
and Corollary 3.2]{Panzer2015}, as done in \cite[\S\S2--3]{Zhou2022mkMpl}. In practice, we can automate the fibration procedures using Au's \texttt{MultipleZetaValues} (v1.2.0) package \cite{Au2022a}.

Concretely speaking, if the output of Au's function \texttt{IterIntDoableQ[\textbraceleft0,1\textbraceright$\cup S\cup $\textbraceleft$ s$\textbraceright]} is a positive integer that  divides $N$ for every point $ s\in S_*$, then we have $ \mathfrak G_{k,k_*}^{(1)}[S,S_*]\subseteq\mathfrak Z_k(N) $.  We will refer to this as ``Au's test'' hereafter.

\begin{table}[t]\centering\caption{Selected CMZV characterizations of  $\mathsf  S_{3,r}\left( a_k;z\right)$ at level $7$, where   $ \eta\colonequals e^{2\pi i/7}$, $ \lambdabar_\nu\colonequals\log(2\sin(\nu\pi))=-\RE\Li_1\left(e^{2\nu\pi i}\right)$, and  $\pL \colonequals\log7=2(\lambdabar_{1/7}+2\lambdabar_{2/7}+\lambdabar_{3/14})= 2(2\lambdabar_{1/7}+\lambdabar_{2/7}-\lambdabar_{1/14})=2(\lambdabar_{1/7}+\lambdabar_{2/7}+\lambdabar_{3/7})\in\mathfrak Z_1(7) $ \label{tab:Zk(7)3k}}

\begin{tiny}\begin{align*}\begin{array}{@{}c@{}l|@{}c@{}l|@{}c@{}l} \hline\hline
&\vphantom{\frac{\int}{1}}\mathsf S_{3,0}\left(\frac1k;\mathfrak r_{1/7}\right)=\int_0^1\Li_1\left( \mathfrak r_{1/7}t ^2(1-t)\right)\frac{\D t}{1-t}&&\vphantom{\frac{\int}{1}}\mathsf S_{3,0}\left(\frac1k;\mathfrak r_{2/7}\right)=\int_0^1\Li_1\left( \mathfrak r_{2/7}t ^2(1-t)\right)\frac{\D t}{1-t}&&\mathsf S_{3,0}\left(\frac1k;\mathfrak r_{3/7}\right)=\int_0^1\Li_1\left( \mathfrak r_{3/7}t ^2(1-t)\right)\frac{\D t}{1-t}\\{}={}&\frac{32 \pi ^2}{147}-2 \big(\lambdabar_{1/7}-\lambdabar_{2/7}\big)^2&{}={}&\frac{2 \pi ^2}{147}-2\lambdabar_{3/14}^{2} &{}={}&\frac{8 \pi ^2}{147}-2\lambdabar_{1/14}^2\\[5pt]&\vphantom{\frac{\int}{1}}\mathsf S_{3,1}\left(\frac1k;\mathfrak r_{1/7}\right)=\int_0^1\Li_2\left( \mathfrak r_{1/7}t ^2(1-t)\right)\frac{\D t}{1-t}&&\vphantom{\frac{\int}{1}}\mathsf S_{3,0}\left(\frac1k;\mathfrak r_{2/7}\right)=\int_0^1\Li_2\left( \mathfrak r_{2/7}t ^2(1-t)\right)\frac{\D t}{1-t}&&\vphantom{\frac{\int}{1}}\mathsf S_{3,0}\left(\frac1k;\mathfrak r_{3/7}\right)=\int_0^1\Li_2\left( \mathfrak r_{3/7}t ^2(1-t)\right)\frac{\D t}{1-t}\\{}={}&12 \RE\Li_{1,1,1}(\eta ^2,1,\eta )&{}={}&-12 \RE\left[\Li_{1,1,1}(\eta ^4,1,\eta )+\Li_{1,1,1}(\eta ^2,1,\eta )\right]&{}={}&12 \RE\Li_{1,1,1}(\eta ^4,1,\eta )\\&{}+ \RE\left[5\Li_3(\eta ^2)+2 \Li_3(\eta )\right]+\frac{48 \zeta (3)}{49}&&{}+ \RE\left[3\Li_3(\eta ^2)+7 \Li_3(\eta )\right]-\frac{366 \zeta (3)}{49}&&{}-\RE\left[8\Li_3(\eta ^2)+9 \Li_3(\eta )\right]-\frac{144 \zeta (3)}{49}\\&{}-12  \big(\lambdabar_{1/7}-\lambdabar_{2/7}\big)\RE\Li_{1,1}(\eta ^2,\eta )&&{} +12 \lambdabar_{3/14} \RE\Li_{1,1}(\eta ^2,\eta) &&{}+12 \lambdabar_{1/14} \RE\Li_{1,1}(\eta ^2,\eta)\\&{} -\frac{2\pi  \I\left[4\Li_2(\eta ^3)-4\Li_2(\eta ^2)+5\Li_2(\eta )\right]}{7}  &&{}-\frac{2\pi  \I\left[\Li_2(\eta ^3)-\Li_2(\eta ^2)-7\Li_2(\eta )\right]}{7} &&{}-\frac{2\pi  \I\left[2\Li_2(\eta ^3)-2\Li_2(\eta ^2)-5\Li_2(\eta )\right]}{7} \\&{}+\frac{2\big(10 \lambdabar_{1/7}^3-3 \lambdabar_{2/7} ^{}\lambdabar_{1/7}^2-15 \lambdabar_{2/7}^2 \lambdabar_{1/7}^{}+11 \lambdabar_{2/7}^3\big)}{3}&&{}-\frac{2\big(6 \lambdabar_{2/7}^3+18 \lambdabar_{3/14}^{} \lambdabar_{2/7}^2-2 \lambdabar_{3/14}^3\big)}{3} &&{}+\frac{2\big(3 \lambdabar_{1/7}^3-18\lambdabar_{1/7}^{} \lambdabar_{2/7}^{} \lambdabar_{1/14}^{} -\lambdabar_{1/14}^3 \big)}{3}  \\&{}-3 \big(\lambdabar_{1/7}-\lambdabar_{2/7}\big){}^2 \pL-\frac{\pi ^2 \big(337 \lambdabar_{1/7}+320 \lambdabar_{2/7}-96 \pL\big)}{294} &&{}-6 \lambdabar_{3/14}^2 \lambdabar_{2/7}^{}+\frac{\pi ^2 \big(24 \lambdabar_{1/7}+21 \lambdabar_{2/7}+23 \lambdabar_{3/14}\big)}{147}&&{}+6\lambdabar_{1/14}^2\lambdabar_{2/7}^{} -\frac{\pi ^2 \big(57 \lambdabar_{1/7}-48 \lambdabar_{2/7}-70 \lambdabar_{1/14}\big)}{294}  \\[3pt]\hline &\vphantom{\frac{\int}{1}}\mathsf S_{3,1}\left(\mathsf H_{2k}-\mathsf H_{3k};\mathfrak r_{1/7}\right)&&\vphantom{\frac{\int}{1}}\mathsf S_{3,1}\left(\mathsf H_{2k}-\mathsf H_{3k};\mathfrak r_{2/7}\right)&&\vphantom{\frac{\int}{1}}\mathsf S_{3,1}\left(\mathsf H_{2k}-\mathsf H_{3k};\mathfrak r_{3/7}\right)\\{}={}&\int_0^1\Li_2\left( \mathfrak r_{1/7}t ^2(1-t)\right)\frac{\log t\D t}{1-t}&{}={}&\int_0^1\Li_2\left( \mathfrak r_{2/7}t ^2(1-t)\right)\frac{\log t\D t}{1-t}&{}={}&\int_0^1\Li_2\left( \mathfrak r_{3/7}t ^2(1-t)\right)\frac{\log t\D t}{1-t}\\{}={}& 10\RE\Li_{1,1,1}(\eta ^2,1,\eta ) &{}={}&-10 \RE\left[\Li_{1,1,1}(\eta ^4,1,\eta )+\Li_{1,1,1}(\eta ^2,1,\eta )\right] &{}={}&10 \RE\Li_{1,1,1}(\eta ^4,1,\eta ) \\&{}+\frac{5\RE\left[5\Li_3(\eta ^2)+2 \Li_3(\eta )\right]}{6}+\frac{40 \zeta (3)}{49}&&{}+\frac{5\RE\left[3\Li_3(\eta ^2)+7 \Li_3(\eta )\right]}{6}-\frac{305 \zeta (3)}{49} &&{}-\frac{5\RE\left[8\Li_3(\eta ^2)+9 \Li_3(\eta )\right]}{6}-\frac{120 \zeta (3)}{49}\\&{}-6   \big(\lambdabar_{1/7}-\lambdabar_{2/7}\big)\RE\Li_{1,1}(\eta ^2,\eta )&&{}+6 \lambdabar_{3/14} \RE\Li_{1,1}(\eta ^2,\eta )&&{}+6 \lambdabar_{1/14} \RE\Li_{1,1}(\eta ^2,\eta)\\&{} -\frac{\pi  \I\left[4\Li_2(\eta ^3)-4\Li_2(\eta ^2)+11\Li_2(\eta )\right]}{7}  &&{}-\frac{\pi  \I\left[\Li_2(\eta ^3)-\Li_2(\eta ^2)-11\Li_2(\eta )\right]}{7} &&{}-\frac{\pi  \I\left[2\Li_2(\eta ^3)-2\Li_2(\eta ^2)-7\Li_2(\eta )\right]}{7}  \\&{}+\frac{2 \lambdabar_{1/7}^3+3 \lambdabar_{2/7}^{} \lambdabar_{1/7}^2-3 \lambdabar_{2/7}^2 \lambdabar_{1/7}^{}+3 \lambdabar_{2/7}^3}{3}&&{}-\frac{\lambdabar_{2/7} \big(10 \lambdabar_{2/7}^2+24 \lambdabar_{3/14}^{} \lambdabar_{2/7}^{}\big)}{3}&&{}+\frac{5 \lambdabar_{1/7}^3-6 \lambdabar_{1/14} \lambdabar_{1/7}^2-18 \lambdabar_{2/7} \lambdabar_{1/14} \lambdabar_{1/7}+\lambdabar_{1/14}^3}{3}  \\&{} -\frac{\big(\lambdabar_{1/7}-\lambdabar_{2/7}\big)^2 \pL}{2}-\frac{\pi ^2 \big(377 \lambdabar_{1/7}-50 \lambdabar_{2/7}-32 \pL\big)}{588} &&{}-5\lambdabar_{3/14}^2 \lambdabar_{2/7}^{}+\frac{\pi ^2 \big(32 \lambdabar_{1/7}+43 \lambdabar_{2/7}+74 \lambdabar_{3/14}\big)}{294}&&{}+5\lambdabar_{1/14}^2\lambdabar_{2/7}^{}-\frac{\pi ^2 \big(31 \lambdabar_{1/7}-16 \lambdabar_{2/7}+60 \lambdabar_{1/14}\big)}{588}
\\[3pt]\hline &\vphantom{\frac{\int}{1}}\mathsf S_{3,1}\left(\mathsf H_{k-1}-\mathsf H_{3k};\mathfrak r_{1/7}\right)&&\vphantom{\frac{\int}{1}}\mathsf S_{3,1}\left(\mathsf H_{k-1}-\mathsf H_{3k};\mathfrak r_{2/7}\right)&&\vphantom{\frac{\int}{1}}\mathsf S_{3,1}\left(\mathsf H_{k-1}-\mathsf H_{3k};\mathfrak r_{3/7}\right)\\{}={}&\int_0^1\Li_2\left( \mathfrak r_{1/7}t ^2(1-t)\right)\frac{\log (1-t)\D t}{1-t}&{}={}&\int_0^1\Li_2\left( \mathfrak r_{2/7}t ^2(1-t)\right)\frac{\log (1-t)\D t}{1-t}&{}={}&\int_0^1\Li_2\left( \mathfrak r_{3/7}t ^2(1-t)\right)\frac{\log (1-t)\D t}{1-t}\\{}={}& 4 \RE\Li_{1,1,1}(\eta ^2,1,\eta )&{}={}&-4  \RE\left[\Li_{1,1,1}(\eta ^4,1,\eta )+\Li_{1,1,1}(\eta ^2,1,\eta )\right]&{}={}&4 \RE\Li_{1,1,1}(\eta ^4,1,\eta )\\{}&+\frac{\RE\left[5\Li_3(\eta ^2)+2 \Li_3(\eta )\right]}{3} +\frac{16 \zeta (3)}{49} &&{}+\frac{ \RE\left[3\Li_3(\eta ^2)+7 \Li_3(\eta )\right]}{3}-\frac{122 \zeta (3)}{49} &&{}-\frac{\RE\left[8\Li_3(\eta ^2)+9 \Li_3(\eta )\right]}{3}-\frac{48 \zeta (3)}{49}   \\&{}-\frac{6\pi\I\Li_2(\eta )}{7}-\frac{2\big(3 \lambdabar_{1/7}^3-9 \lambdabar_{2/7}^2 \lambdabar_{1/7}^{}+5 \lambdabar_{2/7}^3\big)}{3} &&{}+\frac{4\pi\I\Li_2(\eta )}{7}-\frac{2\big(2 \lambdabar_{2/7}^3+3 \lambdabar_{3/14}^{} \lambdabar_{2/7}^2 +2 \lambdabar_{3/14}^3\big)}{3}&&{}+\frac{2\pi\I\Li_2(\eta )}{7}+\frac{2\big(\lambdabar_{1/7}^3-3 \lambdabar_{1/14} \lambdabar_{1/7}^2\big)}{3} \\&{}+ \big(\lambdabar_{1/7}-\lambdabar_{2/7}\big)^2 \pL+\frac{\pi ^2 \big(44 \lambdabar_{1/7}+121 \lambdabar_{2/7}-32 \pL\big)}{294}&&{}-2\lambdabar_{3/14}^2 \lambdabar_{2/7}^{}+\frac{\pi ^2 \big(8 \lambdabar_{1/7}+22 \lambdabar_{2/7}+47 \lambdabar_{3/14}\big)}{294}&&{}+2\lambdabar_{1/14}^2\lambdabar_{2/7}^{}+\frac{\pi ^2 \big(13 \lambdabar_{1/7}-16 \lambdabar_{2/7}-81 \lambdabar_{1/14}\big)}{294} \\[3pt]\hline\hline
\end{array}\end{align*}\end{tiny}\end{table}\begin{table}[p]\caption{Selected CMZV characterizations of  $ \mathsf S_{3,r}\left( a_k;\vartheta^\pm\right)$ at level $8$, where $ \vartheta^\pm\colonequals\pm\frac{\big(\sqrt{2}\pm1\big)^2}{\sqrt{2}}=\frac{4\pm3\sqrt{2}}{2}$, $ \theta\colonequals e^{\pi i/4}$, $\lambda\colonequals\log2 $, $\widetilde \lambda\colonequals\log\big(1+\sqrt{2}\big)$, and $ G\colonequals\I\Li_2(i)$ \label{tab:Zk(8)3k}}

\begin{scriptsize}\begin{align*}\begin{array}{@{}c@{}l|@{}c@{}l} \hline\hline
&\vphantom{\frac{\int}{1}}\mathsf S_{3,0}\left(\frac1k;\mathfrak{r}_{1/8}=\vartheta^+\right)=\int_0^1\Li_1\left( \vartheta^+t ^2(1-t)\right)\frac{\D t}{1-t}&&\vphantom{\frac{\int}{1}}\mathsf S_{3,0}\left(\frac1k;\mathfrak{r}_{3/8}=\vartheta^-\right)=\int_0^1\Li_1\left( \vartheta^-t ^2(1-t)\right)\frac{\D t}{1-t}\\{}={}&\frac{25 \pi ^2}{96}-\frac{(2 \widetilde{\lambda }+\lambda )^2}{8} &{}={}&\frac{\pi ^2}{96}-\frac{(2 \widetilde{\lambda }-\lambda )^2}{8}\\[5pt]&\vphantom{\frac{\int}{1}}\mathsf S_{3,1}\left(\frac1k;\vartheta^+\right)=\int_0^1\Li_2\left( \vartheta^+t ^2(1-t)\right)\frac{\D t}{1-t}&&\vphantom{\frac{\int}{1}}\mathsf S_{3,1}\left(\frac1k;\vartheta^-\right)=\int_0^1\Li_2\left( \vartheta^-t ^2(1-t)\right)\frac{\D t}{1-t}\\{}={}&12 \RE\Li_{1,1,1}(i,1,\theta )+2 \RE\Li_3(\theta )-\frac{15 \zeta (3)}{64}&{}={}&12 \RE\Li_{1,1,1}(i,1,-\theta) -2 \RE\Li_3(\theta )-\frac{9 \zeta (3)}{32}\\&{}+3 (2 \widetilde{\lambda }+\lambda ) \RE\Li_{1,1}(i,\theta ) -3 \pi  \I\Li_2(\theta )+\frac{15 \pi G}{8}&&{}+3 (2 \widetilde{\lambda }-\lambda) \RE\Li_{1,1}(i,\theta )+3 \pi  \I\Li_2(\theta )-\frac{9 \pi  G}{8} \\&{}-\frac{15 \lambda ^2 \widetilde{\lambda }+12 \lambda  \widetilde{\lambda }^2+20 \widetilde{\lambda }^3+\lambda ^3}{24} +\frac{\pi ^2 (136 \widetilde{\lambda }-43 \lambda)  }{192} &&{} -\frac{24 \lambda ^2 \widetilde{\lambda }-12 \lambda  \widetilde{\lambda }^2+32 \widetilde{\lambda }^3-25 \lambda ^3}{48} -\frac{\pi ^2 (17 \widetilde{\lambda }+11 \lambda )}{96} \\[3pt]\hline &\vphantom{\frac{\int}{1}}\mathsf S_{3,1}(\mathsf H_{2k}-\mathsf H_{3k};\vartheta^+)=\int_0^1\Li_1\left( \vartheta^+t ^2(1-t)\right)\frac{\log t\D t}{1-t}&&\vphantom{\frac{\int}{1}}\mathsf S_{3,1}(\mathsf H_{2k}-\mathsf H_{3k};\vartheta^-)=\int_0^1\Li_1\left( \vartheta^-t ^2(1-t)\right)\frac{\log t\D t}{1-t}\\{}={}&10 \RE\Li_{1,1,1}(i,1,\theta ) +\frac{5 \RE\Li_3(\theta )}{3}-\frac{25 \zeta (3)}{128}&{}={}&10 \RE\Li_{1,1,1}(i,1,-\theta )  -\frac{5 \RE\Li_3(\theta )}{3}-\frac{15 \zeta (3)}{64}\\&{}+\frac{3 (2 \widetilde{\lambda }+\lambda ) \RE\Li_{1,1}(i,\theta ) }{2} -\frac{5 \pi  \I\Li_2(\theta )}{2}+\frac{15 \pi  G}{16}&&{}+\frac{3 (2 \widetilde{\lambda }-\lambda) \RE\Li_{1,1}(i,\theta )}{2}+\frac{5\pi  \I\Li_2(\theta )}{2}-\frac{17 \pi  G}{16} \\&{} -\frac{15 \lambda ^2 \widetilde{\lambda }+4 \widetilde{\lambda }^3-2 \lambda ^3}{48} +\frac{5\pi ^2 (28 \widetilde{\lambda }-3 \lambda )}{384} &&{}-\frac{24 \lambda ^2 \widetilde{\lambda }-36 \lambda  \widetilde{\lambda }^2+64 \widetilde{\lambda }^3-31 \lambda ^3}{96} -\frac{\pi ^2 (19 \widetilde{\lambda }+21 \lambda )}{192} \\[3pt]\hline  &\vphantom{\frac{\int}{1}}\mathsf S_{3,1}(\mathsf H_{k-1}-\mathsf H_{3k};\vartheta^+)=\int_0^1\Li_1\left( \vartheta^+t ^2(1-t)\right)\frac{\log(1- t)\D t}{1-t}&&\vphantom{\frac{\int}{1}}\mathsf S_{3,1}(\mathsf H_{k-1}-\mathsf H_{3k};\vartheta^-)=\int_0^1\Li_1\left( \vartheta^-t ^2(1-t)\right)\frac{\log(1-t)\D t}{1-t}\\[5pt]{}={}&4 \RE\Li_{1,1,1}(i,1,\theta )+\frac{2 \RE\Li_3(\theta )}{3}-\frac{5 \zeta (3)}{64}-\pi  \I\Li_2(\theta )&{}={}&4 \RE\Li_{1,1,1}(i,1,-\theta )-\frac{2 \RE\Li_3(\theta )}{3}-\frac{3 \zeta (3)}{32}+\pi  \I\Li_2(\theta )-\frac{\pi  G}{2}\\&{}-\frac{6 \lambda ^2 \widetilde{\lambda }-12 \lambda  \widetilde{\lambda }^2-24 \widetilde{\lambda }^3-5 \lambda ^3}{96} -\frac{\pi ^2 (46 \widetilde{\lambda }-3 \lambda)}{384} &&{}+\frac{12 \lambda ^2 \widetilde{\lambda }+24 \lambda  \widetilde{\lambda }^2-48 \widetilde{\lambda }^3+10 \lambda ^3}{192} -\frac{\pi ^2 (2 \widetilde{\lambda }+21 \lambda )}{384} \\[3pt]\hline\hline
\end{array}\end{align*}\end{scriptsize}\end{table}
\begin{proof}[Proof of Corollary \ref{cor:3kCMZV}]
Feeding Theorem \ref{thm:3kGPL}(a)  to Au's test, while setting $x$ as $ -1$, $ -\frac{1\pm\sqrt{5}}{2}$, $-2$, $-\big(1\pm\sqrt{2}\big) $, $ -\left(\frac{1\pm\sqrt{5}}{2}\right)^2$,  and $ -\big(1\pm \sqrt{3}\big)$,  respectively, we receive in return{\allowdisplaybreaks\begin{align}\small\begin{array}{@{}r@{{}\in{}}lr@{{}\in{}}lr@{{}\in{}}l}\mathsf S_{3,r}\left(a_k;\frac{1}{2}\right)&\mathfrak Z_{r+2}(4),& \mathsf S_{3,r}\left(a_k;\frac{1\pm\sqrt{5}}{2}\right)&\mathfrak Z_{r+2}(5),&
\mathsf S_{3,r}\left(a_k;\frac{2^{3}}{3}\right)&\mathfrak Z_{r+2}(6)\\\mathsf S_{3,r}\left(a_k;\pm\frac{\big(\sqrt{2}\pm1\big)^2}{\sqrt{2}}\right)&\mathfrak Z_{r+2}(8),&\mathsf S_{3,r}\left(a_k;\pm\frac{1}{\sqrt{5}}\left(\frac{1\pm\sqrt{5}}{2}\right)^{5}\right)&\mathfrak Z_{r+2}(10),&\mathsf S_{3,r}\left(a_k;2\big(1\pm\sqrt{3}\big)\right)&\mathfrak Z_{r+2}(12),\end{array}\label{eq:3kZ(4-12)}
\end{align}}when $ (a_k)_{k\in\mathbb Z_{>0}}\in\big\{\big(\frac1k\big)_{k\in\mathbb Z_{>0}},$  $(\mathsf H_{k}-\mathsf H_{3k})_{k\in\mathbb Z_{>0}},$ $(\mathsf H_{2k}-\mathsf H_{3k})_{k\in\mathbb Z_{>0}}\big\}$. For $ x=-3$, we get
  \eqref{eq:3kZ(6)'}.

We note that  the formulae in   \eqref{eq:3kZ(4-12)} already entail \eqref{eq:3kZ(N)}  with $ N\in\{4$, $5$, $6$, $8$, $10$, $12\}$. For $ N\in\{7,9\}$, simply run Au's test again. To move onto generic positive integers $N$, note that \begin{align}\begin{split}&
1-\mathfrak r_\nu t^2(1-t)\\={}&\left[ 1+t\frac{\big(e^{2\nu\pi i  }-e^{2\pi i/3}\big) \big(e^{2\nu\pi i  }-e^{-2\pi i/3}\big) }{e^{2\nu\pi i  }} \right]\left[ 1-t\frac{\big(e^{2\nu\pi i  }-e^{2\pi i/3}\big) \big(e^{2\nu\pi i  }-e^{-2\pi i/3}\big) }{\big(e^{2\nu\pi i  }+1\big)e^{2\nu\pi i  }} \right]\\{}&\times\left[ 1-t\frac{\big(e^{2\nu\pi i  }-e^{2\pi i/3}\big) \big(e^{2\nu\pi i  }-e^{-2\pi i/3}\big) }{e^{2\nu\pi i  }+1} \right],
\end{split}
\end{align}  so  a fibration procedure akin to the proof of Theorem \ref{thm:3kGPL} leaves us \begin{align}
\mathsf S_{3,r}\left( a_k ;\mathfrak r_{m/N}\right)\in\mathfrak G_{r+2;r+2}^{(e^{2\pi i m/N})}\left[ \big\{e^{2\pi i/3},e^{-2\pi i/3},-1\big\}; \big\{0,1,e^{2\pi i/3},e^{-2\pi i/3},-1\big\}\right].
\end{align}In view of the fact that \cite[(1.3$''$)]{Zhou2023SunCMZV} \begin{align}\mathfrak
Z_{k}(N')=\Span_{\mathbb Q}\left\{G(z_1,\dots,z_k;z)\left|\begin{smallmatrix}z_1^{N'},\dots,z_{k}^{N'}\in\{0,1\}\\z^{N'}=1;z_1\neq z\end{smallmatrix}\right.\right\}\tag{\ref{eq:Zk(N)_defn}$'$}\label{eq:Zk(N)_defn'}
\end{align}for  $ N'=\lcm(6,N)\in\mathbb Z_{\geq3}$, we have   \eqref{eq:3kZ(N)'}.
\end{proof}

\begin{table}[p]\caption{Selected CMZV characterizations of  $ \mathsf S_{3,r}\left(a_k;\kappa^\pm\right)$  at level $10$, where $ \kappa^\pm\colonequals\pm\frac{1}{\sqrt{5}}\left(\frac{1\pm\sqrt{5}}{2}\right)^{5}=\frac{25\pm11\sqrt{5}}{10}$, $ \varsigma\colonequals e^{2\pi i/5}$, $ \lambda\colonequals\log2$, $ \pd\colonequals\log \phi$, and $ \mathscr L\colonequals\log5$\label{tab:Zk(10)3k}}

\begin{scriptsize}\begin{align*}\begin{array}{@{}c@{}l|@{}c@{}l}\hline\hline
&\mathsf S_{3,0}\left(\frac1k;\mathfrak{r}_{1/10}=\kappa^+\right)=\int_0^1\Li_1\left( \kappa^+t ^2(1-t)\right)\frac{\D t}{1-t}&&\mathsf S_{3,0}\left(\frac1k;\mathfrak{r}_{3/10}=\kappa^-\right)=\int_0^1\Li_1\left( \kappa^-t ^2(1-t)\right)\frac{\D t}{1-t}\vphantom{\frac{\int}{\int}}\\{}={}&\frac{49 \pi ^2}{150}-\frac{(2 \pd+\mathscr{L} )^2}{8}&{}={}&\frac{\pi ^2}{150}-\frac{(2 \pd-\mathscr{L} )^2}{8} \\[5pt]&\mathsf S_{3,1}\left(\frac1k;\kappa^+\right)=\int_0^1\Li_2\left( \kappa^+t ^2(1-t)\right)\frac{\D t}{1-t}&&\mathsf S_{3,1}\left(\frac1k;\kappa^-\right)=\int_0^1\Li_2\left( \kappa^-t ^2(1-t)\right)\frac{\D t}{1-t}\\{}={}&12 \RE\Li_{1,1,1}\left(\varsigma ,1,-\frac{1}{\varsigma ^2}\right)+\frac{22\RE\Li_3\left(-\frac{1}{\varsigma ^2}\right)}{5} -\frac{234 \zeta (3)}{125}&{}={}&-12 \RE\Li_{1,1,1}\left(\varsigma ,1,-\frac{1}{\varsigma ^2}\right)+\frac{19\RE\Li_3\left(-\frac{1}{\varsigma ^2}\right)}{5}-\frac{1008 \zeta (3)}{125}  \\&{}+3 (2 \pd +\mathscr{L}) \RE\Li_{1,1}\left(\varsigma ,-\frac{1}{\varsigma ^2}\right)-\frac{\pi  \I\big[15 \Li_2\big(\varsigma ^2\big)+11\Li_2(\varsigma )\big]}{10}  &&{} -6 \RE\Li_{1,1,1}\left(-\varsigma ,1,-\frac{1}{\varsigma ^2}\right)+3 (2 \pd -\mathscr{L}) \RE\Li_{1,1}\left(\varsigma ,-\frac{1}{\varsigma ^2}\right)  \\&{}-\frac{12 \pd  \mathscr{L}^2+60 \pd ^2 \mathscr{L}+160 \pd ^3-\mathscr{L}^3}{48} +\frac{\pi ^2 (652 \pd -49 \mathscr{L})}{300} &&{}+\frac{\pi  \I\big[23 \Li_2\big(\varsigma ^2\big)+15\Li_2(\varsigma )\big]}{10}+\frac{12 \pd  \mathscr{L}^2-60 \pd ^2 \mathscr{L}+112 \pd ^3+\mathscr{L}^3}{48}\\&&&{} +\frac{\pi ^2 (8 \pd -\mathscr{L})}{300} \\[3pt]\hline \vphantom{\frac{\int}{\int}}&\mathsf S_{3,1}\left(\mathsf H_{2k}-\mathsf H_{3k};\kappa^+\right)=\int_0^1\Li_1\left( \kappa^+t ^2(1-t)\right)\frac{\log t\D t}{1-t}&&\mathsf S_{3,1}\left(\mathsf H_{2k}-\mathsf H_{3k};\kappa^-\right)=\int_0^1\Li_1\left( \kappa^-t ^2(1-t)\right)\frac{\log t\D t}{1-t}\\{}={}& 10 \RE\Li_{1,1,1}\left(\varsigma ,1,-\frac{1}{\varsigma ^2}\right)+\frac{11\RE\Li_3\left(-\frac{1}{\varsigma ^2}\right)}{3} -\frac{39 \zeta (3)}{25}&{}={}&-10 \RE\Li_{1,1,1}\left(\varsigma ,1,-\frac{1}{\varsigma ^2}\right)+\frac{19\RE\Li_3\left(-\frac{1}{\varsigma ^2}\right)}{6} -\frac{168 \zeta (3)}{25}\\&{}+\frac{3(2 \pd +\mathscr{L}) \RE\Li_{1,1}\left(\varsigma ,-\frac{1}{\varsigma ^2}\right)}{2} -\frac{\pi\I\big[39 \Li_2\big(\varsigma ^2\big)+23\Li_2(\varsigma )\big]}{20}&&{}-5 \RE\Li_{1,1,1}\left(-\varsigma ,1,-\frac{1}{\varsigma ^2}\right)+\frac{3(2 \pd -\mathscr{L}) \RE\Li_{1,1}\left(\varsigma ,-\frac{1}{\varsigma ^2}\right)}{2}    \\&{} -\frac{\pd  \big(44 \pd ^2+3 \mathscr{L}^2-12 \pd  \mathscr{L}\big)}{48} +\frac{\pi ^2 (245 \pd +6 \mathscr{L})}{300} &&{}+\frac{\pi  \I\big[39\Li_2\big(\varsigma ^2\big)+23\Li_2(\varsigma )\big]}{20} +\frac{\pd  \big(100 \pd ^2+3 \mathscr{L}^2-36 \pd  \mathscr{L}\big)}{48}\\&&&{} -\frac{\pi ^2 (25 \pd -19 \mathscr{L})}{300} \\[3pt]\hline\vphantom{\frac{\int}{\int}}&\mathsf S_{3,1}\left(\mathsf H_{k-1}-\mathsf H_{3k};\kappa^+\right)=\int_0^1\Li_1\left( \kappa^+t ^2(1-t)\right)\frac{\log(1- t)\D t}{1-t}&&\mathsf S_{3,1}\left(\mathsf H_{k-1}-\mathsf H_{3k};\kappa^-\right)=\int_0^1\Li_1\left( \kappa^-t ^2(1-t)\right)\frac{\log (1-t)\D t}{1-t}\\{}={}& 4 \RE\Li_{1,1,1}\left(\varsigma ,1,-\frac{1}{\varsigma ^2}\right)+\frac{22\RE\Li_3\left(-\frac{1}{\varsigma ^2}\right)}{15}-\frac{78 \zeta (3)}{125}&{}={}&-4 \RE\Li_{1,1,1}\left(\varsigma ,1,-\frac{1}{\varsigma ^2}\right)+\frac{19\RE\Li_3\left(-\frac{1}{\varsigma ^2}\right)}{15} -\frac{336 \zeta (3)}{125}\\&{}-\frac{3\pi\I\big[2 \Li_2\big(\varsigma ^2\big)+\Li_2(\varsigma )\big]}{5}  +\frac{36 \pd ^2 \mathscr{L}+32 \pd ^3-\mathscr{L}^3}{48}   &&{}-2 \RE\Li_{1,1,1}\left(-\varsigma ,1,-\frac{1}{\varsigma ^2}\right)+\frac{2\pi\I\big[2 \Li_2\big(\varsigma ^2\big)+ \Li_2(\varsigma )\big]}{5}  \\&{}-\frac{\pi ^2 (65 \pd-3\mathscr L )}{150}&&{}-\frac{\mathscr{L}^3+12 \pd ^2 \mathscr{L}-48 \pd ^3}{48} -\frac{\pi ^2 (28 \pd-19 \mathscr{L} )}{300} \\[3pt]\hline\hline
\end{array}\end{align*}\end{scriptsize}\end{table}\begin{table}[ht]\caption{Selected CMZV characterizations of  $ \mathsf S_{3,r}\left(a_k;\mu^\pm\right)$  at level $12$, where $ \mu^\pm\colonequals2\big(1\pm \sqrt{3}\big)$,  $ \omega\colonequals e^{2\pi i/3}$,  $ \varrho\colonequals e^{\pi i/3}$, $ \lambda\colonequals\log2$, $ \widetilde\varLambda\colonequals\log\big(2+\sqrt{3}\big)$, and $ G\colonequals\I\Li_2(i)$\label{tab:Zk(12)3k}}

\begin{scriptsize}\begin{align*}\begin{array}{@{}c@{}l|@{}c@{}l}\hline\hline
&\mathsf S_{3,0}\left(\frac1k;\mathfrak{r}_{1/12}=\mu^+\right)=\int_0^1\Li_1\left( \mu^+t ^2(1-t)\right)\frac{\D t}{1-t}&&\mathsf S_{3,0}\left(\frac1k;\mathfrak{r}_{5/12}=\mu^-\right)=\int_0^1\Li_1\left( \mu^-t ^2(1-t)\right)\frac{\D t}{1-t}\vphantom{\frac{\int}{\int}}\\{}={}&\frac{3 \pi ^2}{8}-\frac{\widetilde\varLambda^2}{2}&{}={}&\frac{\pi ^2}{24}-\frac{\widetilde\varLambda^2}{2}\\[5pt]&\mathsf S_{3,1}\left(\frac1k;\mu^+\right)=\int_0^1\Li_2\left( \mu^+t ^2(1-t)\right)\frac{\D t}{1-t}&&\mathsf S_{3,1}\left(\frac1k;\mu^-\right)=\int_0^1\Li_2\left( \mu^-t ^2(1-t)\right)\frac{\D t}{1-t}\\{}={}& 12 \RE\Li_{1,1,1}\left(\varrho ,1,\frac{i}{\varrho }\right)+\frac{19 \zeta (3)}{16}&{}={}&12 \RE\Li_{1,1,1}\left(\varrho ,1,-\frac{i}{\varrho }\right)+\frac{19 \zeta (3)}{16}\\&{}+6 \widetilde{\varLambda } \RE\Li_{1,1}\left(\varrho ,\frac{i}{\varrho }\right)+\frac{21\pi  \I\Li_2(\omega )}{16} -\frac{19 \pi  G}{6} &&{}+6 \widetilde{\varLambda } \RE\Li_{1,1}\left(\varrho ,\frac{i}{\varrho }\right)-\frac{9\pi  \I\Li_2(\omega )}{16}+\frac{17 \pi  G}{6}   \\[5pt]&{}-\frac{\widetilde{\varLambda }^2 \big(\widetilde{\varLambda }+9 \lambda \big)}{12} +\frac{\pi ^2 \big(25 \widetilde{\varLambda }+27 \lambda \big)}{48} &&{}+\frac{\widetilde{\varLambda }^2 \big(\widetilde{\varLambda }-9 \lambda\big)}{12} -\frac{\pi ^2 \big(19 \widetilde{\varLambda }-3 \lambda\big)}{48} \\[5pt]\hline \vphantom{\frac{\int}{\int}}&\mathsf S_{3,1}\left(\mathsf H_{2k}-\mathsf H_{3k};\mu^+\right)=\int_0^1\Li_1\left( \mu^+t ^2(1-t)\right)\frac{\log t\D t}{1-t}&&\mathsf S_{3,1}\left(\mathsf H_{2k}-\mathsf H_{3k};\mu^-\right)=\int_0^1\Li_1\left( \mu^-t ^2(1-t)\right)\frac{\log t\D t}{1-t}\\{}={}&10 \RE\Li_{1,1,1}\left(\varrho ,1,\frac{i}{\varrho }\right)+\frac{95 \zeta (3)}{96}&{}={}&10 \RE\Li_{1,1,1}\left(\varrho ,1,-\frac{i}{\varrho }\right)+\frac{95 \zeta (3)}{96}\\&{}+3 \widetilde{\varLambda } \RE\Li_{1,1}\left(\varrho ,\frac{i}{\varrho }\right)+\frac{5\pi  \I\Li_2(\omega )}{32} -\frac{89 \pi  G}{36}&&{}+3 \widetilde{\varLambda } \RE\Li_{1,1}\left(\varrho ,\frac{i}{\varrho }\right)-\frac{25\pi  \I\Li_2(\omega )}{32}+\frac{83 \pi  G}{36} \\&{}+\frac{\widetilde{\varLambda }^2 \big(\widetilde{\varLambda }-3 \lambda \big)}{24} +\frac{\pi ^2 \big(109 \widetilde{\varLambda }+27 \lambda \big)}{288} &&{}-\frac{\widetilde{\varLambda }^2 \big(\widetilde{\varLambda }+3 \lambda \big)}{24} -\frac{\pi ^2 \big(91 \widetilde{\varLambda }-3 \lambda \big)}{288} \\[3pt]\hline \vphantom{\frac{\int}{\int}}&\mathsf S_{3,1}\left(\mathsf H_{k-1}-\mathsf H_{3k};\mu^+\right)=\int_0^1\Li_1\left( \mu^+t ^2(1-t)\right)\frac{\log(1- t)\D t}{1-t}&&\mathsf S_{3,1}\left(\mathsf H_{k-1}-\mathsf H_{3k};\mu^-\right)=\int_0^1\Li_1\left( \mu^-t ^2(1-t)\right)\frac{\log (1-t)\D t}{1-t}\\{}={}&4 \RE\Li_{1,1,1}\left(\varrho ,1,\frac{i}{\varrho }\right)+\frac{19 \zeta (3)}{48}-\frac{\pi  \I\Li_2(\omega )}{2} -\frac{8 \pi  G}{9}&{}={}&4 \RE\Li_{1,1,1}\left(\varrho ,1,-\frac{i}{\varrho }\right)+\frac{19 \zeta (3)}{48}-\frac{\pi  \I\Li_2(\omega )}{2} +\frac{8 \pi  G}{9}\\&{}+\frac{\lambda  \widetilde{\varLambda }^2}{4}-\frac{\pi ^2 \big(10 \widetilde{\varLambda }+27 \lambda \big)}{144} &&{}+\frac{\lambda  \widetilde{\varLambda }^2}{4}-\frac{\pi ^2 \big(14 \widetilde{\varLambda }+3 \lambda \big)}{144} \\[3pt]\hline\hline
\end{array}\end{align*}\end{scriptsize}\end{table}

\begin{remark}Some special values of $\mathsf  S_{3,r}\big(a_k;\mathfrak r_{1/4}=\frac{1}{2}\big)$ have already appeared in \cite[Remark 1.4]{Zhou2023SunCMZV}. We illustrate the $ N\in\{5$, $6$, $7$, $8$, $10$, $12\}$ situations of the corollary above by Tables \ref{tab:Zk(5)3k}--\ref{tab:Zk(12)3k}.\footnote{The notations we pick for these tables are intended to be mnemonic.  For example, to represent roots of unity, the Greek letters $ \omega$, $\varsigma$, $\varrho $, $\eta$, and  $\theta$ are chosen for their visual similarities to the Arabic numerals $3$, $5$, $6$, $7$, and $8$.   The  letters $ \vartheta$, $\kappa $, and $\mu$ are selected because they occupy the 8th, 10th, and 12th positions in the Greek alphabet.} The current lack of basis reduction for $ \mathfrak Z_3(9)$ in Au's  \texttt{MultipleZetaValues}  package prevents us from going beyond the following examples at level $9$:\begin{align}\mathsf S_{3,0}\left( \frac{1}{k} ;\mathfrak r_{1/9}\right)
={}&\frac{8 \pi ^2}{27}-2 \log ^2\left(2 \cos\frac{\pi }{9}\right),\\\mathsf S_{3,0}\left( \frac{1}{k} ;\mathfrak r_{2/9}\right)={}&\frac{2 \pi ^2}{27}-2 \log ^2\left(2 \cos \frac{2 \pi }{9}\right),\\\mathsf S_{3,0}\left( \frac{1}{k} ;\mathfrak r_{4/9}\right)={}&\frac{2 \pi ^2}{27}-2 \log ^2\left(2 \cos \frac{4 \pi }{9}\right).
\end{align}Notably, the tabulated entries for $ \mathsf S_{3,0}\big(\frac1k;z\big)$  (as well as the last three displayed formulae) conform to the  description in Theorem \ref{thm:conj1.2}.

All these tabulated entries are manifestly members of the respective CMZV spaces, thanks to Goncharov's filtration property     \cite[\S1.2]{Goncharov1998}:\begin{align}\mathfrak Z_j(N)\mathfrak
Z_{ k}(N)\subseteq \mathfrak
Z_{j+ k}(N),\quad \text{for all } j,k\in \mathbb Z_{\geq0},\label{eq:Zk(N)_filt}
\end{align}
 together with the observations that {\allowdisplaybreaks\begin{align}
  \log2={}&-\Li_1(-1)\in\mathfrak Z_1(2),\\\log 3={}& -2\RE\Li_1\big(e^{2\pi i/3}\big)\in\mathfrak Z_1(3),\\\log\frac{1+\sqrt{5}}{2}={}&\RE \left[\Li_1\big(e^{{2 \pi  i}/{5}}\big)- \Li_1\big(e^{{4 \pi  i}/{5}}\big)\right]\in \mathfrak Z_1(5),\\\log5={}&-2\RE \left[\Li_1\big(e^{{2 \pi  i}/{5}}\big)+ \Li_1\big(e^{{4 \pi  i}/{5}}\big)\right]\in \mathfrak Z_1(5),\\\log\big(1+\sqrt{2}\big)={}&\RE\left[\Li_1\big( e^{\pi i/4}\big)-\Li_1\big( e^{3\pi i/4}\big)\right]\in \mathfrak Z_1(8),\\ \log\big(2+\sqrt{3}\big)={}&2\RE\Li_1\big(e^{\pi i/6}\big)\in \mathfrak Z_1(12),
\end{align}}and  $ \pi i\in\mathfrak Z_1(N)$ for $ N\in\mathbb Z_{\geq3}$ \cite[Lemma 4.1]{Au2022a}.
\eor\end{remark}

 \begin{table}[p]\caption{Selected CMZV characterizations of infinite series at levels $ 4$, $5$, $6$, $7$, $8$, $9$, $10$, and $12$, where $\phi\colonequals \frac{1+\sqrt{5}}{2}$, $ \omega\colonequals e^{2\pi i/3}$, $ \varsigma\colonequals e^{2\pi i/5}$, $ \varrho\colonequals e^{2\pi i/6}$, $\eta\colonequals e^{2\pi i/7} $,  $\theta\colonequals e^{2\pi i/8} $, $   \qoppa\colonequals e^{2\pi i/9} $, $ \lambda\colonequals\log2$, $\varLambda \colonequals\log3$, $\pd\colonequals \log \phi$, $ \pL\colonequals \log 7$, $ \widetilde{\lambda}\colonequals\log\big(1+\sqrt2\big)$, $ \mathscr L\colonequals \log5$,  $\widetilde\varLambda \colonequals\log\big(2+\sqrt3\big)$,  $ s_\nu\colonequals \sin(\nu\pi)$,  $ c_\nu\colonequals \cos(\nu\pi)$, $ \lambdabar_{\nu}\colonequals \log(2s_\nu)$, and $ G\colonequals \I\Li_2(i)$\label{tab:3k+1_3k+2}}
\begin{tiny}\begin{align*}\begin{array}{c|c|l|l}\hline\hline N&m&\displaystyle\sum_{k=0}^\infty\frac{\mathsf H_k-\mathsf H_{3k+1}}{(3k+1)\binom{3k}k}\mathfrak r_{m/N}^k=\int_{0}^1\frac{\mathfrak \log (1-t)\D t}{1-\mathfrak r_{m/N}t^{2}(1-t)}&\displaystyle\sum_{k=1}^\infty\frac{\mathsf H_k-\mathsf H_{3k-1}}{(2k-1)\binom{3k}k}\mathfrak r_{m/N}^k=\frac{2}{3}\int_{0}^1\frac{\mathfrak r_{m/N}(1-t)\log (1-t)\D t}{1-\mathfrak r_{m/N}t^{2}(1-t)}\\\hline 4&1&\frac{\lambda ^2-\frac{5\pi ^2}{24}}{5}-\frac{2(2 G)}{5}\vphantom{\frac\int\int}&\frac{2\left( \lambda ^2-\frac{5\pi ^2}{24} \right)}{15}+\frac{2 G}{15}\\\hline 5&1&\frac{\big(3\sqrt{5}-1\big)\left( \pd ^2-\frac{29 \pi ^2}{150} \right)}{44}-\frac{\sqrt{2} \big(5 \sqrt{5}+2\big)\left\{\I\left[\Li_2(\varsigma )+2\Li_2\big(\varsigma ^2\big)\right]+\frac{4 \pi  \pd }{5}\right\}}{22 \sqrt{5+\sqrt{5}}} &\frac{\big(2 \sqrt{5}+3\big) \left( \pd ^2-\frac{29 \pi ^2}{150} \right)}{33}+\frac{\sqrt{2} \big(5 \sqrt{5}-9\big)\left\{\I\left[\Li_2(\varsigma )+2\Li_2\big(\varsigma ^2\big)\right]+\frac{4 \pi  \pd}{5} \right\}}{66 \sqrt{\sqrt{5}+5}}  \vphantom{\frac{\frac\int\int}\int}\\[3pt]&2&-\frac{\big(3\sqrt{5}+1\big)\left( \pd ^2+\frac{49 \pi ^2}{150} \right)}{44}+\frac{\sqrt{2} \big(5 \sqrt{5}-2\big)\left\{-\I\left[2\Li_2(\varsigma )-\Li_2\big(\varsigma ^2\big)\right]+\frac{2 \pi  \pd }{5}\right\}}{22 \sqrt{5-\sqrt{5}}}&-\frac{\big(2 \sqrt{5}-3\big)\left( \pd ^2+\frac{49 \pi ^2}{150} \right)}{33}-\frac{\sqrt{2} \big(5 \sqrt{5}+9\big)\left\{-\I\left[2\Li_2(\varsigma )-\Li_2\big(\varsigma ^2\big)\right]+\frac{2 \pi  \pd }{5} \right\}}{66 \sqrt{5-\sqrt{5}}}\\[5pt]\hline6&1&\frac{3\left[6\RE\Li_{1,1}(\omega ,\varrho )-3 \lambda  \varLambda +\varLambda ^2-\frac{\pi ^2}{6}\right]}{28} -\frac{5\sqrt{3}\left[\frac{5\I\Li_2(\omega )}{2}+\pi\lambda\right]}{28}&\frac{2\left[6\RE\Li_{1,1}(\omega ,\varrho )-3 \lambda  \varLambda +\varLambda ^2-\frac{\pi ^2}{6}\right]}{7}-\frac{2\left[\frac{5\I\Li_2(\omega )}{2}+\pi\lambda\right]}{21 \sqrt{3}}\vphantom{\frac{\frac11}1}\\[3pt]\hline 7&1&\frac{2c^{2}_{1/7}\big[  6\RE\Li_{1,1}(\eta^2,\eta)-5 \lambdabar_{1/7}^2+10 \lambdabar_{2/7}^2+\lambdabar_{1/7} \big(3 \pL-8 \lambdabar_{2/7}\big)-3 \lambdabar_{2/7} \pL-\frac{43 \pi ^2}{588}\big]}{(1+2c_{2/7})(5+4c_{2/7})}&\frac{4c^{2}_{1/7}\big[  6\RE\Li_{1,1}(\eta^2,\eta)-5 \lambdabar_{1/7}^2+10 \lambdabar_{2/7}^2+\lambdabar_{1/7} \big(3 \pL-8 \lambdabar_{2/7}\big)-3 \lambdabar_{2/7} \pL-\frac{43 \pi ^2}{588}\big]}{3(5+4c_{2/7})/(1+2c_{2/7})}\\&&{}-\frac{c_{1/7}( 2+c_{2/7} )\left\{\I\left[\Li_2(\eta)+\Li_2(\eta^{2})-\Li_2(\eta^{3})\right]-\frac{4\pi  (4 \lambdabar_{1/7}+2\lambdabar_{2/7}-\pL)}{7} \right\}}{s_{1/7}(1+2c_{2/7})(5+4c_{2/7})}&{}-\frac{(2c_{2/7}+c_{4/7})\left\{\I\left[\Li_2(\eta)+\Li_2(\eta^{2})-\Li_2(\eta^{3})\right]-\frac{4\pi  (4 \lambdabar_{1/7}+2\lambdabar_{2/7}-\pL)}{7}\right\}}{3s_{2/7}(5+4c_{2/7})/(1+2c_{2/7})}\\&2&\frac{2c^{2}_{2/7}\big[6\RE\Li_{1,1}(\eta^2,\eta)-3 \lambdabar_{2/7}^2+\big(-2 \lambdabar_{1/7}-4 \lambdabar_{2/7}+\pL\big)^2-\frac{79 \pi ^2}{588}  \big]}{(1+2c_{4/7})(5+4c_{4/7})}&\frac{4c^{2}_{2/7}\big[  6\RE\Li_{1,1}(\eta^2,\eta)-3 \lambdabar_{2/7}^2+\big(-2 \lambdabar_{1/7}-4 \lambdabar_{2/7}+\pL\big)^2-\frac{79 \pi ^2}{588}\big]}{3(5+4c_{4/7})/(1+2c_{4/7})}\\&&{}-\frac{c_{2/7}( 2+c_{4/7} )\left\{\I\left[\Li_2(\eta)+\Li_2(\eta^{2})-\Li_2(\eta^{3})\right]+\frac{2\pi  (\lambdabar_{1/7}-\lambdabar_{2/7})}{7} \right\}}{s_{2/7}(1+2c_{4/7})(5+4c_{4/7})}&{}-\frac{(2c_{4/7}+c_{8/7})\left\{\I\left[\Li_2(\eta)+\Li_2(\eta^{2})-\Li_2(\eta^{3})\right]+\frac{2\pi  (\lambdabar_{1/7}-\lambdabar_{2/7})}{7}\right\}}{3s_{4/7}(5+4c_{4/7})/(1+2c_{4/7})}\\&3&\frac{2c^{2}_{3/7}\big[  6\RE\Li_{1,1}(\eta^2,\eta)+7 \lambdabar_{1/7}^2-2 \lambdabar_{1/7} \left(\lambdabar_{2/7}+\pL\right)+\frac{ (\pL-2 \lambdabar_{2/7})^2}{4}+\frac{281 \pi ^2}{588}\big]}{(1+2c_{6/7})(5+4c_{6/7})}&\frac{4c^{2}_{3/7}\big[  6\RE\Li_{1,1}(\eta^2,\eta)+7 \lambdabar_{1/7}^2-2 \lambdabar_{1/7} \left(\lambdabar_{2/7}+\pL\right)+\frac{ (\pL-2 \lambdabar_{2/7})^2}{4}+\frac{281 \pi ^2}{588}\big]}{3(5+4c_{6/7})/(1+2c_{6/7})}\\&&{}-\frac{c_{3/7}( 2+c_{6/7} )\left\{-\I\left[\Li_2(\eta)+\Li_2(\eta^{2})-\Li_2(\eta^{3})\right]+\frac{2\pi  (\pL-2 \lambdabar_{1/7}-4 \lambdabar_{2/7})}{7} \right\}}{s_{3/7}(1+2c_{6/7})(5+4c_{6/7})}&-\frac{(2c_{6/7}+c_{12/7})\left\{-\I\left[\Li_2(\eta)+\Li_2(\eta^{2})-\Li_2(\eta^{3})\right]+\frac{2\pi  (\pL-2 \lambdabar_{1/7}-4 \lambdabar_{2/7})}{7}\right\}}{3s_{6/7}(5+4c_{6/7})/(1+2c_{6/7})}\\\hline8&1&\frac{\big(5 \sqrt{2}-4\big)\left[6\RE\Li_{1,1}(i,\theta )-\frac{20 \widetilde{\lambda }^2-4 \lambda  \widetilde{\lambda }+5 \lambda ^2}{16} +\frac{5 \pi ^2}{192}\right]}{34} +\frac{\big(3 \sqrt{2}-16\big)\left( \frac{3 G}{2}+\frac{5 \pi  \widetilde{\lambda }}{4} \right)}{34} &\frac{\big(7 \sqrt{2}+8\big)\left[6\RE\Li_{1,1}(i,\theta )-\frac{20 \widetilde{\lambda }^2-4 \lambda  \widetilde{\lambda }+5 \lambda ^2}{16} +\frac{5 \pi ^2}{192}\right]}{51} -\frac{2\big(3 \sqrt{2}+1\big)\left( \frac{3 G}{2}+\frac{5 \pi  \widetilde{\lambda }}{4} \right)}{51}  \\[3pt]&3&-\frac{\big(5 \sqrt{2}+4\big)\left[-6\RE\Li_{1,1}(i,\theta )+\frac{4 \widetilde{\lambda }^2-4 \lambda  \widetilde{\lambda }+ 13\lambda ^2}{16} +\frac{23 \pi ^2}{192}\right]}{34}+\frac{\big(3 \sqrt{2}+16\big)\left( - \frac{3 G}{2}+\frac{\pi  \widetilde{\lambda }}{4} \right)}{34} &-\frac{\big(7 \sqrt{2}-8\big)\left[-6\RE\Li_{1,1}(i,\theta )+\frac{4 \widetilde{\lambda }^2-4 \lambda  \widetilde{\lambda }+ 13\lambda ^2}{16} +\frac{23 \pi ^2}{192}\right]}{51}-\frac{2\big(3 \sqrt{2}-1\big)\left(- \frac{3 G}{2}+\frac{\pi  \widetilde{\lambda }}{4} \right)}{51} \\[3pt]\hline9&1&\frac{2c^{2}_{1/9}\big[ 6 \RE\Li_{1,1}\big(\qoppa ^2,\qoppa \big)+\lambdabar_{1/9} \big(3 \varLambda -8 \lambdabar_{2/9}\big)-3 \lambdabar_{2/9} \varLambda +\lambdabar_{1/9}^2+4 \lambdabar_{2/9}^2+\frac{13 \pi ^2}{108} \big]}{(1+2c_{2/9})(5+4c_{2/9})}&\frac{4c^{2}_{1/9}\big[ 6 \RE\Li_{1,1}\big(\qoppa ^2,\qoppa \big)+\lambdabar_{1/9} \big(3 \varLambda -8 \lambdabar_{2/9}\big)-3 \lambdabar_{2/9} \varLambda +\lambdabar_{1/9}^2+4 \lambdabar_{2/9}^2+\frac{13 \pi ^2}{108} \big]}{3(5+4c_{2/9})/(1+2c_{2/9})}\\&&{}-\frac{c_{1/9}( 2+c_{2/9} )\left\{\I\left[\Li_2(\qoppa )+\Li_2\big(\qoppa ^2\big)-\Li_2(\omega )\right]+\frac{2 \pi(\varLambda-2  \lambdabar_{1/9})}{3}\right\}}{s_{1/9}(1+2c_{2/9})(5+4c_{2/9})}&{}-\frac{(2c_{2/9}+c_{4/9})\left\{\I\left[\Li_2(\qoppa )+\Li_2\big(\qoppa ^2\big)-\Li_2(\omega )\right]+\frac{2 \pi(\varLambda-2  \lambdabar_{1/9})}{3}\right\}}{3s_{2/9}(5+4c_{2/9})/(1+2c_{2/9})}\\&2&\frac{2c^{2}_{2/9}\big[ -6 \RE\Li_{1,1}\big(\omega,\qoppa ^{2}\big)-5 \lambdabar_{2/9} \varLambda -4 \lambdabar_{1/9} \big(\varLambda -4 \lambdabar_{2/9}\big)+4 \lambdabar_{1/9}^2+13 \lambdabar_{2/9}^2+\varLambda ^2-\frac{17 \pi ^2}{108} \big]}{(1+2c_{4/9})(5+4c_{4/9})}&\frac{4c^{2}_{2/9}\big[ -6 \RE\Li_{1,1}\big(\omega,\qoppa ^{2}\big)-5 \lambdabar_{2/9} \varLambda -4 \lambdabar_{1/9} \big(\varLambda -4 \lambdabar_{2/9}\big)+4 \lambdabar_{1/9}^2+13 \lambdabar_{2/9}^2+\varLambda ^2-\frac{17 \pi ^2}{108} \big]}{3(5+4c_{4/9})/(1+2c_{4/9})}\\{}&&{}-\frac{c_{2/9}( 2+c_{4/9} )\left\{\I\left[-\Li_2(\qoppa )+2\Li_2\left(\qoppa ^2\right)+\frac{4}{3}\Li_2(\omega )\right]+\frac{\pi  (\varLambda -2 \lambdabar_{2/9})}{3} \right\}}{s_{2/9}(1+2c_{4/9})(5+4c_{4/9})}&{}-\frac{( 2c_{4/9}+c_{8/9} )\left\{\I\left[-\Li_2(\qoppa )+2\Li_2\left(\qoppa ^2\right)+\frac{4}{3}\Li_2(\omega )\right]+\frac{\pi  (\varLambda -2 \lambdabar_{2/9})}{3} \right\}}{3s_{4/9}(5+4c_{4/9})/(1+2c_{4/9})}\\&4&\frac{2c^{2}_{4/9}\big[ 6 \RE\Li_{1,1}\left(\omega,\qoppa\right)-\lambdabar_{2/9} \varLambda +\lambdabar_{1/9} \big(4 \lambdabar_{2/9}-5 \varLambda \big)+7 \lambdabar_{1/9}^2+\lambdabar_{2/9}^2+\frac{\varLambda ^2}{4}+\frac{19 \pi ^2}{36} \big]}{(1+2c_{8/9})(5+4c_{8/9})}&\frac{4c^{2}_{4/9}\big[ 6 \RE\Li_{1,1}\left(\omega,\qoppa\right)-\lambdabar_{2/9} \varLambda +\lambdabar_{1/9} \big(4 \lambdabar_{2/9}-5 \varLambda \big)+7 \lambdabar_{1/9}^2+\lambdabar_{2/9}^2+\frac{\varLambda ^2}{4}+\frac{19 \pi ^2}{36} \big]}{3(5+4c_{8/9})/(1+2c_{8/9})}\\{}&&{}-\frac{c_{4/9}( 2+c_{8/9} )\left\{\I\left[-2\Li_2(\qoppa )+\Li_2\left(\qoppa ^2\right)-\frac{2}{3}\Li_2(\omega )\right]-\frac{2\pi  (\lambdabar_{1/9}+\lambdabar_{2/9})}{3}   \right\}}{s_{4/9}(1+2c_{8/9})(5+4c_{8/9})}&-\frac{( 2c_{8/9}+c_{16/9}  )\left\{\I\left[-2\Li_2(\qoppa )+\Li_2\left(\qoppa ^2\right)-\frac{2}{3}\Li_2(\omega )\right]-\frac{2\pi  (\lambdabar_{1/9}+\lambdabar_{2/9})}{3}   \right\}}{3s_{8/9}(5+4c_{8/9})/(1+2c_{8/9})}\\\hline 10&1&\frac{\big(35-11 \sqrt{5}\big)\left[6\RE\Li_{1,1}\left(\varsigma ,-\frac{1}{\varsigma ^2}\right)+\frac{\mathscr{L}^2-2 \pd  \mathscr{L}-20 \pd ^2}{4} +\frac{31 \pi ^2}{150}\right]}{124}   \vphantom{\frac{\frac{\frac11}\int}1}&\frac{ \big(7 \sqrt{5}+20\big)\left[6\RE\Li_{1,1}\left(\varsigma ,-\frac{1}{\varsigma ^2}\right)+\frac{\mathscr{L}^2-2 \pd  \mathscr{L}-20 \pd ^2}{4}+\frac{31 \pi ^2}{150}\right]}{93}\\[2pt]&&{}+\frac{\sqrt{5+2 \sqrt{5}} \big(29 \sqrt{5}-81\big)\left\{\frac{1}{2}\I\left[\Li_2(\varsigma )+3\Li_2\big(\varsigma ^2\big)\right]+\frac{14 \pi  \pd }{5}\right\}}{124}&{}-\frac{ \sqrt{5+2 \sqrt{5}}\big(19 \sqrt{5}+10\big)\left\{\frac{1}{2}\I\left[\Li_2(\varsigma )+3\Li_2\big(\varsigma ^2\big)\right]+\frac{14 \pi  \pd }{5}\right\}}{465} \\[2pt]&3&\frac{ \big(35+11 \sqrt{5}\big)\left[-6\RE\Li_{1,1}\left(\varsigma ,-\frac{1}{\varsigma ^2}\right)+\frac{\mathscr{L}^2+2 \pd  \mathscr{L}+4 \pd ^2}{4}-\frac{28 \pi ^2}{75} \right]}{124} &-\frac{ \big(7 \sqrt{5}-20\big)\left[-6\RE\Li_{1,1}\left(\varsigma ,-\frac{1}{\varsigma ^2}\right)+\frac{\mathscr{L}^2+2 \pd  \mathscr{L}+4 \pd ^2}{4}-\frac{28 \pi ^2}{75} \right]}{93} \\&&{}-\frac{ \sqrt{5-2 \sqrt{5}}\big(29 \sqrt{5}+81\big)\left\{\frac{1}{2}\I\left[3\Li_2(\varsigma )-\Li_2\big(\varsigma ^2\big)\right]-\frac{2 \pi  \pd}{5} \right\}}{124} &{}+\frac{ \sqrt{5-2 \sqrt{5}}\big(19 \sqrt{5}-10\big)\left\{\frac{1}{2}\I\left[3\Li_2(\varsigma )-\Li_2\big(\varsigma ^2\big)\right]-\frac{2 \pi  \pd}{5} \right\}}{465} \\[3pt]\hline12&1&\frac{\big(3 \sqrt{3}-1\big) \big[6\RE\Li_{1,1}\big(\varrho ,\frac{i}{\varrho }\big)-\frac{\widetilde{\varLambda } (6 \lambda -\widetilde{\varLambda })}{4}+\frac{17 \pi ^2}{48}\big]}{52} &\frac{\big(5 \sqrt{3}+7\big)\big[6\RE\Li_{1,1}\big(\varrho ,\frac{i}{\varrho }\big)-\frac{\widetilde{\varLambda } (6 \lambda -\widetilde{\varLambda })}{4}+\frac{17 \pi ^2}{48}\big]}{39}  \\&&{} -\frac{\big(11 \sqrt{3}+5\big) \big[-\frac{G}{3}+\frac{15\I\Li_2(\omega )}{8}+\frac{3\pi  (\lambda+\widetilde{\varLambda })}{4} \big]}{52} &{}-\frac{\big(\sqrt{3}+17\big)\big[-\frac{G}{3}+\frac{15\I\Li_2(\omega )}{8}+\frac{3\pi  (\lambda+\widetilde{\varLambda })}{4} \big] }{39}  \\[3pt]&5&-\frac{\big(3 \sqrt{3}+1\big) \big[-6\RE\Li_{1,1}\big(\varrho ,\frac{i}{\varrho }\big)+\frac{\widetilde{\varLambda } (6 \lambda +\widetilde{\varLambda })}{4}+\frac{ \pi ^2}{16}\big]}{52} &-\frac{\big(5 \sqrt{3}-7\big)\big[-6\RE\Li_{1,1}\big(\varrho ,\frac{i}{\varrho }\big)+\frac{\widetilde{\varLambda } (6 \lambda +\widetilde{\varLambda })}{4}+\frac{ \pi ^2}{16}\big]}{39}{}\\&&{}+\frac{\big(11 \sqrt{3}-5\big)  \big[-\frac{G}{3}-\frac{15\I\Li_2(\omega )}{8}-\frac{\pi  (\lambda-\widetilde{\varLambda })}{4} \big]}{52} &{}+\frac{\big(\sqrt{3}-17\big) \big[-\frac{G}{3}-\frac{15\I\Li_2(\omega )}{8}-\frac{\pi  (\lambda-\widetilde{\varLambda })}{4} \big]}{39}   \\[3pt]\hline\hline
\end{array}\end{align*}\end{tiny}
\end{table}

\begin{remark}By default,  Au's \texttt{MultipleZetaValues}  package uses  spanning sets for  CMZV spaces that are different from our practices  in Tables \ref{tab:Zk(5)3k}--\ref{tab:Zk(12)3k}. Sometimes, Au's spanning set  may lead to bulky expressions, like the right-hand side of the equation below (cf.\  Table \ref{tab:Zk(10)3k}):{\allowdisplaybreaks\begin{align}
\begin{split}
\RE\Li_{1,1,1} \left(\varsigma ,1,-\frac{1}{\varsigma ^2}\right)={}&\frac{\Li_3\big(\!-\!\frac{1}{4}\big)}{48}+\frac{\Li_3\big(\frac{1}{5}\big)}{16}-\frac{\Li_3\left(\frac{3-\sqrt{5}}{4} \right)}{24} -\frac{5 \Li_3\big(\sqrt{5}-2\big)}{72}+\frac{53\Li_3\left(\frac{\sqrt{5}-1}{2} \right)}{400} \\&{}-\frac{\Li_3\left(\frac{\sqrt{5}-1}{4} \right)}{8} -\frac{\Li_3\left(\frac{1}{\sqrt{5}}\right)}{2}-\frac{\Li_3\big(2 \sqrt{5}-4\big)}{24} -\frac{119 \zeta (3)}{750}\\&{}+\frac{3\pi  \I\Li_2(\varsigma )}{20} +\frac{3\pi  \I\Li_2\big(\varsigma^{2}\big)}{10} +\frac{5 \lambda ^2 \pd }{48}-\frac{\lambda  \pd ^2}{6}+\frac{1091 \pd ^3}{3600}\\{}&+\frac{\pd  \mathscr{L}^2}{32}-\frac{\pd ^2 \mathscr{L}}{8}-\frac{191 \pi ^2 \pd }{4500}-\frac{17 \pi ^2 \mathscr{L}}{2400}-\frac{\pi ^2 \lambda }{144},
\end{split}
\end{align}}where  $ \varsigma\colonequals e^{2\pi i/5}$, $ \lambda\colonequals\log2$, $ \pd\colonequals\log \frac{1+\sqrt{5}}{2}$, and $ \mathscr L\colonequals\log5$.

One can establish the equivalence between  our explicit CMZV\ formulae  in Tables \ref{tab:Zk(5)3k}--\ref{tab:Zk(12)3k} and the outputs from Au's \texttt{MZIntegrate}  function, by the \texttt{MZExpand} command in the  \texttt{MultipleZetaValues}  package. For example, to verify  the entry for   $\mathsf S_{3,1}\left(\mathsf H_{k-1}-\mathsf H_{3k};\frac{2^3}{3}\right)$ in Table \ref{tab:Zk(6)3k}, one types \begin{quote}\texttt{MZIntegrate[PolyLog[1, (2\^{}3/3)*t\^{}2*(1 - t)]*(Log[1 - t]/(1 - t)), \{t, 0, 1\}] -
  MZExpand[4*ReColoredMZV[6, \{1, 1, 1\}, \{2, 0, 1\}] + Zeta[3]/18 -
    Pi*ImColoredMZV[3, \{2\}, \{1\}] + (Log[3]\^{}2*(3*Log[2] - Log[3]))/6 -    \linebreak  (Pi\^{}2*(3*Log[2] - 2*Log[3]))/18]}\end{quote}in \texttt{Mathematica} and checks that the output is  $0$.
 \eor\end{remark}

 \begin{remark}In Table \ref{tab:3k+1_3k+2}, we represent some special cases for \eqref{eq:Hk-H3k(3k+1)} and \eqref{eq:Hk-H3k(2k-1)} via $ \overline{\mathbb Q}$-linear combinations of CMZVs in $ \mathfrak Z_2(N)$, where $ N\in\{4,5,6,7,8,9,10,12\}$. Again, the \texttt{MZExpand} command is instrumental during the construction of these tabulated entries. \eor\end{remark}\begin{table}[t]\caption{Selected CMZV characterizations of  $\mathsf S_{4,r}\big(\mathsf H_k;z\big)$ at levels $8$ and $12$, where $ \omega\colonequals e^{2\pi i/3}$,  $ \varrho\colonequals e^{\pi i/3}$,   $ \theta\colonequals e^{\pi i/4}$,  $\lambda\colonequals\log2 $, $\varLambda\colonequals\log3 $, $\widetilde \lambda\colonequals\log\big(1+\sqrt{2}\big)$,  $ \widetilde\varLambda\colonequals\log\big(2+\sqrt{3}\big)$, and $ G\colonequals\I\Li_2(i)$\label{tab:Zk(8,12)4k}}

\begin{scriptsize}\begin{align*}\begin{array}{@{}c@{}l}\hline\hline \vphantom{\frac{\frac\int\int}{\int}}&{}\mathsf S_{4,1}\big(\mathsf H_k,2^4\big)=2\int_0^1\left[\Li_2\left(\big[2^{2}t(1-t)\big]^{2}\right)+\frac{\log^2\left( 1-\left[2^{2}t(1-t)\right]^{2}\right)}{2}\right]\frac{\D t}{1-t}\\{}={}&-384 \RE\left[\Li_{1,1,1}(i,1,\theta )+\Li_{1,1,1}(i,1,-\theta )\right]-196 \zeta (3)-256 \widetilde{\lambda } \RE\Li_{1,1}(i,\theta ) +48 \pi  G\\&{}+2 \big(12 \lambda ^2 \widetilde{\lambda }-4 \lambda  \widetilde{\lambda }^2+16 \widetilde{\lambda }^3-7 \lambda ^3\big)-\frac{\pi ^2 \big(4 \widetilde{\lambda }-63 \lambda\big)}{6}\in\mathfrak Z_3(8)\\[3pt]\hline
\vphantom{\frac{\frac1\int}{\int}}&{}\mathsf S_{4,1}\big(\mathsf H_k,2^2\big)=2\int_0^1\left[\Li_2\left([2t(1-t)]^{2}\right)+\frac{\log^2\left( 1-[2t(1-t)]^{2}\right)}{2}\right]\frac{\D t}{1-t}\\{}={}&48 \RE\left[\Li_{1,1,1}\left(\varrho ,1,\frac{i}{\varrho }\right)+\Li_{1,1,1}\left(\varrho ,1,-\frac{i}{\varrho }\right)\right]-8 \RE\left[\Li_{1,1,1}\left(\omega ,1,\frac{i}{\varrho }\right)+\Li_{1,1,1}\left(\omega ,1,-\frac{i}{\varrho }\right)\right]+\frac{383 \zeta (3)}{9}\\&{}+32 \widetilde{\varLambda } \RE\Li_{1,1}\left(\varrho ,\frac{i}{\varrho }\right)-4 \widetilde{\varLambda } \RE\Li_{1,1}\left(\omega ,\frac{i}{\varrho }\right)-11 \pi  \I\Li_2(\omega )+\frac{4 \pi  G}{3}-\frac{\widetilde{\varLambda }^2 \big(4 \lambda -\widetilde{\varLambda }\big)}{2} +\frac{\pi ^2 \big(-109 \widetilde{\varLambda }+36 \lambda +10 \varLambda \big)}{72} \in\mathfrak Z_3(12) \\[3pt]\hline\hline
\end{array}\end{align*}\end{scriptsize}\end{table}

Thus far, we have concluded our deduction of Corollary \ref{cor:3kCMZV} from Theorem \ref{thm:3kGPL}, which has been accompanied by  some  tabulated special values.
 Next, we will perform a similar service on Theorem \ref{thm:4kGPL}. \begin{proof}[Proof of Corollary \ref{cor:4kCMZV}]The statements follow from routine applications of Au's test to Theorem \ref{thm:4kGPL}(a), upon specializing $x$ to $\frac12$ and $ e^{\pi i/3 }$, respectively.
\end{proof}\begin{remark}For $r=1$, we illustrate  Corollary \ref{cor:4kCMZV} with Table \ref{tab:Zk(8,12)4k}. Apart from these tabulated examples, we also have \begin{align}
\begin{split}&
\sum_{k=1}^{\infty}\frac{\mathsf H_k }{(4k+1)\binom{4k}{2k}}2^{2k}=\int_0^1\frac{\Li_1\left([2t(1-t)]^{2}\right)}{1-[2t(1-t)]^{2}}\D t\\={}&\frac{ \RE\left[4\Li_{1,1}\left(\varrho ,\frac{i}{\varrho }\right)-\Li_{1,1}\left(\omega ,\frac{i}{\varrho }\right)\right]}{\sqrt{3}}+\frac{4 G}{3}-\frac{\lambda  \widetilde{\varLambda }}{\sqrt{3}}+\frac{\widetilde{\varLambda }^2}{8 \sqrt{3}}-\frac{\varLambda  \widetilde{\varLambda }}{4 \sqrt{3}}-\frac{\pi  \lambda }{2}+\frac{29 \pi ^2}{96 \sqrt{3}}
\end{split}\intertext{and}\begin{split}&
\sum_{k=1}^{\infty}\frac{\mathsf H_k }{(4k+3)\binom{4k}{2k}}2^{2k}=\int_0^1\frac{\Li_1\left([2t(1-t)]^{2}\right)}{1-[2t(1-t)]^{2}}[16 t^2 (1-t)-1]\D t\\={}&-\frac{ 5\RE\left[4\Li_{1,1}\left(\varrho ,\frac{i}{\varrho }\right)-\Li_{1,1}\left(\omega ,\frac{i}{\varrho }\right)\right]}{\sqrt{3}}+4 G+\frac{5 \lambda  \widetilde{\varLambda }}{\sqrt{3}}-\frac{5 \widetilde{\varLambda }^2}{8 \sqrt{3}}+\frac{5 \varLambda  \widetilde{\varLambda }}{4 \sqrt{3}}-\frac{3 \pi  \lambda }{2}-\frac{145 \pi ^2}{96 \sqrt{3}}
\end{split}
\end{align}as special cases of Theorem \ref{thm:4kGPL}(b).
\eor\end{remark}

\end{document}